\documentclass{amsart}
\usepackage{amssymb}
\usepackage{amsmath} 
\usepackage{amscd}
\usepackage{amsbsy}
\usepackage{commath, longtable}
\usepackage{comment, enumerate}
\usepackage{geometry}
\usepackage[matrix,arrow]{xy}
\usepackage{hyperref}
\usepackage{mathrsfs}
\usepackage{color}
\usepackage{mathtools,caption}
\usepackage[table,dvipsnames]{xcolor}
\usepackage{tikz-cd}
\usepackage{longtable}
\usepackage[utf8]{inputenc}
\usepackage[OT2,T1]{fontenc}
\usepackage[normalem]{ulem}
\usepackage{hyperref}
\usepackage[customcolors]{hf-tikz}
\hypersetup{
 colorlinks=true,
 linkcolor=DarkOrchid,
 filecolor=blue,
 citecolor=olive,
 urlcolor=orange,
 pdftitle={RNT Project},
 }
\usepackage{booktabs}

\DeclareSymbolFont{cyrletters}{OT2}{wncyr}{m}{n}
\DeclareMathSymbol{\Sha}{\mathalpha}{cyrletters}{"58}

\DeclareMathOperator{\Norm}{\mathsf{N}}
\DeclareMathOperator{\rk}{rk}
\DeclareMathOperator{\val}{val}
\DeclareMathOperator{\Hom}{Hom}

\DeclareMathOperator{\Gal}{Gal}

\DeclareMathOperator{\Cl}{Cl}

\DeclareMathOperator{\GL}{GL}

\DeclareMathOperator{\rank}{rk}

\DeclareMathOperator{\Sym}{Sym}

\DeclareMathOperator{\Res}{Res}
\DeclareMathOperator{\Aut}{Aut}

\DeclareMathOperator{\Image}{Image}

\DeclareMathOperator{\ur}{ur}
\DeclareMathOperator{\Id}{Id}

\newcommand{\QQ}{\mathbb Q}
\newcommand{\ZZ}{\mathbb Z}

\newcommand{\PP}{\mathbb P}

\newcommand{\Fp}{\mathbb{F}_p}
\newcommand{\Qp}{\mathbb{Q}_p}

\newcommand{\cF}{\mathcal F}

\newcommand{\cO}{\mathcal O}
\newcommand{\cCL}{\mathcal{C}_L}
\newcommand{\cCF}{\mathcal{C}_F}
\newcommand{\fa}{\mathfrak a}

\newcommand{\fN}{\mathfrak N}
\newcommand{\fm}{\mathfrak m}
\newcommand{\fn}{\mathfrak n}

\newcommand{\fp}{\mathfrak p}

\newcommand{\fq}{\mathfrak q}

\newcommand{\fP}{\mathfrak P}
\newcommand{\fQ}{\mathfrak Q}

\newcommand{\cH}{\mathcal{H}}

\newcommand{\cG}{\mathcal{G}}

\newtheorem*{cor*}{Corollary}
\newtheorem*{theorem*}{Theorem}
\newtheorem*{conj*}{Conjecture}
\newtheorem*{Ques*}{Question}
\newtheorem{Th}{Theorem}[section]
\newtheorem{Lemma}[Th]{Lemma}
\newtheorem{prop}[Th]{Proposition}
\newtheorem{cor}[Th]{Corollary}

\newtheorem{lthm}{Theorem}

\definecolor{Green}{rgb}{0.0, 0.5, 0.0}

\theoremstyle{definition}
\newtheorem{definition}[Th]{Definition}
\newtheorem{rem}[Th]{Remark}
\newtheorem*{Remark*}{Remark}

\makeatletter
\theoremstyle{plain} 
\newtheorem*{intr@thm}{\intr@thmname}

\newtheorem*{c@njecture}{\conjn@name}
\newcommand{\myl@bel}[2]{%
 \protected@write \@auxout {}{\string \newlabel {#1}{{#2}{\thepage}{#2}{#1}{}} }%
 \hypertarget{#1}{}
   } 
    {
        \def\conjn@name{#2}
        \begin{c@njecture}[{#1}]\myl@bel{#3}{#2}
    }
    {
        \end{c@njecture}
    }
\makeatother

\begin{document}
\title[]{On the $p$-ranks of class groups of certain Galois extensions}

\author[U.~Asarhasa]{Ufuoma Asarhasa}
\address[Asarhasa]{Penn State University, State College, PA}
\email{uva5039@psu.edu}

\author[R.~Gambheera]{Rusiru Gambheera}
\address[Gambheera]{University of California, Santa Barbara, CA}
\email{rusiru@ucsb.edu}

\author[D.~Kundu]{Debanjana Kundu}
\address[Kundu]{University of Texas Rio Grande Valley, Edinburg, TX}
\email{dkundu@math.toronto.edu}

\author[E.~Nunez Lon-wo]{Enrique Nunez Lon-wo}
\address[Nunez Lon-wo]{University of Toronto, Toronto, ON, Canada}
\email{enrique.nunezlon.wo@mail.utoronto.ca}

\author[A.~Sheth]{Arshay Sheth}
\address[Sheth]{
University of Warwick, Coventry, UK} 
\email{arshay.sheth@warwick.ac.uk}

\date{\today}


\begin{abstract}
Let $p$ be an odd prime, let $N$ be a prime with $N \equiv 1 \pmod{p}$, and let $\zeta_p$ be a primitive $p$-th root of unity.
We study the $p$-rank of the class group of $\QQ(\zeta_p, N^{1/p})$ using Galois cohomological methods and obtain an exact formula for the $p$-rank in terms of the dimensions of certain Selmer groups.
Using our formula, we provide a numerical criterion to establish upper and lower bounds for the $p$-rank, analogous to the numerical criteria provided by F.~Calegari--M.~Emerton and K.~Schaefer--E.~Stubley for the $p$-ranks of the class group of $\QQ(N^{1/p})$.
In the case $p=3$, we use Redei matrices to provide a numerical criterion to exactly calculate the $3$-rank, and also study the distribution of the $3$-ranks as $N$ varies through primes which are $4,7 \pmod{9}$.
\end{abstract}

\maketitle
\setcounter{tocdepth}{2}
\tableofcontents

\section{Introduction}

\subsubsection*{\sc{Background and motivation}}
Let $p$ be an odd prime.
The study of the $p$-torsion of the class groups of cyclotomic fields has a rich history, dating back to the work of E.~Kummer in the middle of the nineteenth century.
He showed that  $p$ divides the class number of $\QQ(\zeta_p)$ if and only if there exists an even positive integer $k$ with $2 \leq k \leq p-3$ such that $p$ divides the numerator of the $k$-th Bernoulli number $B_k$.
The primes that satisfy the above condition are called irregular, and primes which are not irregular are called regular.
The celebrated Herbrand--Ribet theorem, which we now proceed to briefly recall, provides a refinement to Kummer's criterion.
Let $\Delta=\Gal(\QQ(\zeta_p)/\QQ)$ and note that $\Delta$ acts on the class group $A:=\Cl(\QQ(\zeta_p))$.
Let $C=A/A^p$; then $C$ is an $\mathbb F_p$-vector space and is also equipped with an induced action of $\Delta$.
Let $\chi$ denote the mod $p$ cyclotomic character
\[
\chi: \Delta \xrightarrow{ \simeq } \mathbb F_p^{\times}
\]
defined via $\sigma(\zeta_p)= \zeta_p^{\chi(\sigma)}$. We consider the decomposition
\[
C= \bigoplus_{i=1}^{p-1} C(\chi^{i}),
\]
where $C( \chi^{i}  )= \{c \in C: \sigma \cdot c= \chi^{i}(\sigma) \cdot c \text{ for all } \sigma \in \Delta \}$.
The Herbrand--Ribet theorem (see \cite[Theorem~1.1]{ribet1976modular}) states that if $k$ is a positive even integer with $2 \leq k \leq p-3$, then $C(\chi^{1-k}) \neq 0$  if and only if $p$ divides the numerator of $B_k$. 

The dimension of $C$ as an $\Fp$-vector space is called the $p$-rank of $\Cl(\QQ(\zeta_p))$.
Thus, both Kummer's criterion and the Herbrand--Ribet theorem can be regarded as assertions about the $p$-rank of $\Cl(\QQ(\zeta_p))$.
Let $N>1$ be a positive integer; it is natural to consider the more general problem of investigating the $p$-ranks of the class groups of the number field $F=\QQ(N^{1/p})$; the case $N=1$ being the subject of the above results.
In this direction, progress has been made in the case that $N$ is prime and $N \equiv 1 \pmod{p}$; we restrict to this case in the discussion below and denote the $p$-rank of the class group of a number field $K$ by $\rk_p(\Cl(K))$. 
An argument via genus theory can be used to show that we always have $\rk_p(\Cl(F)) \geq 1$. 
In \cite{Calegari-Emerton}, using techniques from the deformation theory of Galois representations, F.~Calegari--M.~Emerton proved that if $p \geq 5$ and if $M = \prod_{k=1}^{ \frac{N-1}{2}} k^k$  is a $p$-th power in $\mathbb F_N^{\times}$, then $\rk_p(\Cl(F))) \geq 2$.
E.~Lecouturier noticed that the converse of the Calegari--Emerton result fails when $p = 7, N = 337$; see \cite{lecouturier2018galois}. 
Recently, P.~Wake--C.~Wang-Erikson  \cite{wake2020rank} gave a new proof of the theorem of Calegari--Emerton by showing that $M$ being a $p$-th power in $\mathbb F_N^{\times}$ is equivalent to the vanishing of a certain cup product in Galois cohomology.
Building on the techniques of Wake--Wang-Erikson, the $p$-rank distribution of the class group of $\mathbb Q(N^{1/p})$ was subsequently studied in great detail by K.~Schaefer--E.~Stubley in \cite{SS19}.
To explain their results, for odd $i$ in the range $1 \leq i \leq p-4$, let 
\[
M_i= \prod_{k=1}^{N-1} \prod_{a=1}^{k-1} k^{a^i}.
\]
The tuple $(p, -i)$ is called a regular pair if the $\chi^{-i}$ eigenspace of $\Cl(\mathbb Q(\zeta_p))$ is trivial. 
Schaefer--Stubley prove an upper bound
\[
\rk_p( \Cl(F) ) \leq \rk_p(  \Cl(\mathbb Q(\zeta_p)))+p-2-2\mu, 
\]
where $\mu$ is the number of odd $i$ in the range $1 \leq i \leq p-4$ such that $(p, -i)$ is a regular pair and $M_i$ is not a $p$-th power in $\mathbb F_N^{\times}$.
Using this upper bound, they show that if $p$ is regular and $\rk_p(F) \geq 2$,  then at least one of the $M_i$ is a $p$-th power in $\mathbb F_N^{\times}$.  
Since $M$  is a $p$-th power in $\mathbb F_N^{\times}$ if and only if $M_1$ is, their result can be regarded as a partial converse to theorem of Calegari--Emerton.

\subsubsection*{\sc{Our results and proof techniques}}
Motivated by the above works, we investigate the $p$-ranks of the class group of the Galois closure of $\mathbb Q(N^{1/p})$. 
Specifically, we study the $p$-rank of the class group of the number field $L = \QQ(\zeta_p, N^{1/p})$ where $p$ is an odd prime and $N$ is a prime with $N \equiv 1\pmod{p}$. 
For a regular prime $p$, we prove the following inequalities
in Theorems~\ref{lower bound theorem} and \ref{cft ub theorem},
\begin{equation}
\label{bounds}
\frac{p-1}{2} \leq \rk_p(\Cl(L)) \leq (p-1)(p-2).
\end{equation}
The lower bound is independent of the regularity hypothesis. 
Our theorem(s) can handle the upper bound even when $p$ is irregular, but the bounds are weaker.

A priori, it is not clear whether all the values in the above range can actually be attained.
Even if they are attained, do they occur infinitely often?
What is the $p$-rank distribution?

In this paper, we make a modest attempt at addressing these questions using two different approaches, namely, class field theory (and Redei matrices) and Galois cohomology.
First, we focus on the results and predictions obtained via class field theory.
In the second half, we shift our focus towards cohomological methods to refine \eqref{bounds}.

\subsubsection*{\underline{\emph{Prediction}}}
In the simplest case that $p=3$, results in \cite[Section~3]{SS19} assert that the $3$-rank of the class group of $F=\QQ(N^{1/3})$ is \emph{always} 1.
By \eqref{bounds}, the $3$-rank of $\Cl(L)$ is either 1 or 2.
Computer experiments done via SAGE suggested the following $3$-rank distribution of $\Cl(L)$ when $N$ is varied over all primes of the form $1\pmod{3}$,
\[
\PP\left(\rk_3(\Cl(L)) = 1\right) = \frac{2}{3} \textrm{ and } \PP\left(\rk_3(\Cl(L)) = 2\right) = \frac{1}{3}.
\]
Here, we write $\PP\left(\rk_3(\Cl(L)) = r\right)$ to mean
\[
\lim_{x\rightarrow \infty} \frac{\# \{N \leq x \ : \ N \equiv 1\pmod{3} \textrm{ is a prime and } \rk_3(\Cl(L))=r\}}{\# \{N \leq x \ : \ N \equiv 1\pmod{3} \textrm{ is a prime}\}}.
\]
Moreover, the distribution remains the same when restricted to the cases $N \equiv 1\pmod{9}$ and $N\equiv 4,7\pmod{9}$.

\subsubsection*{Results obtained via class field theory and Redei matrices}

\subsubsection*{\underline{\emph{Progress towards the prediction}}}
We provide a characterization in terms of divisibility conditions to ascertain when the $3$-rank of $\Cl(L)$ is 1 or 2. 
When $N\equiv 1\pmod{3}$, it is possible to write 
\[
4N = A^2+27B^2
\]
for some integers $A,B$ (unique up to sign).
We then prove the following result.
\begin{lthm}[Theorem~\ref{criterion}]
\label{thm A}
With notation as above
\begin{enumerate}[\textup{(}i\textup{)}]
\item if $N\not\equiv 1 \pmod{9}$, then $\rk_3\left(\Cl(L)\right)=2$ if and only if $3\mid B$.
\item if $N \equiv 1 \pmod{9}$, then $\rk_3\left(\Cl(L)\right)=2$ if and only if $A$ is a 9th power modulo $N$.
\end{enumerate}
\end{lthm}

In Section~\ref{sec: N 4,7 mod 9}, using Theorem~\ref{thm A} and tools from class field theory, we prove the following result.
\begin{lthm}
\label{Theorem B}
Set $L = \QQ(\zeta_3, N^{1/3})$.
As $N$ varies over all primes of the form $4,7\pmod{9}$,
\[
\PP\left(\rk_3(\Cl(L)) = 1\right) = \frac{2}{3} \textrm{ and } \PP\left(\rk_3(\Cl(L)) = 2\right) = \frac{1}{3}.
\]
\end{lthm}

On the other hand, when $N$ varies over primes of the form $1 \pmod{9}$, the analysis is more difficult.
In Section~\ref{sec: 1 mod 9} we provide heuristic reasons for the above proportions to hold, but are unable to provide a rigorous proof.
The key difficulty lies in the fact that when $N\equiv 1 \pmod{9}$ the ambiguous classes (see Definition~\ref{amb class and str amb class}) are not always strictly ambiguous.
This can be reinterpreted in term of whether $\zeta_3$ is a unit norm or not, but calculating statistics for this description was not possible.
Contrary to our initial expectation, when $N\equiv 1 \pmod{9}$ it is more frequent that the ambiguous classes are \textit{not} strongly ambiguous.
When the coincidence occurs, it follows that $\rk_3(\Cl(L))$ is always 2; see Lemma~\ref{coincidence impies rank 2}.
But when there is non-coincidence of the two classes, both possibilities arise.
One of our main tools is the machinery developed by F.~Gerth to calculate the $3$-rank of $\Cl(L)$ using the rank of certain (Redei) matrices whose entries are determined by cubic Hilbert symbols.
However, when there is non-coincidence of the two classes calculating some of the entries of the matrix is rather abstract and can not be made precise. 

\subsubsection*{Results obtained via Galois cohomology}

The initial goal of the project was to completely determine the $p$-rank distribution of the class group of $L$.
This is a rather difficult problem; building on the work of Schaefer--Stubley \cite{SS19} we prove results in this direction which is explained below.

\subsubsection*{\underline{\emph{Some abstract but precise results}}}

The first main result is Theorem~\ref{main theorem Rusiru note 3}.
The exact statement is technical and would require us to introduce a lot of notation, which we avoid here.
In words, the theorem provides a concrete description of how many unramified Galois extensions with a specified Galois group the number field $L$ can have.
For a precise count of such extensions see Corollary~\ref{define rij}.
This count depends on the dimension of certain well-defined Selmer groups (Galois cohomology groups) with coefficients in abstract Galois modules.

Next, we prove an exact formula for the $p$-rank of the class group of $L$ as sums of these cohomology dimensions.
This is done by studying a more intricate problem which is the structure of $\Cl(L) \otimes \Fp$ as a $\Gal(L/\QQ)$-module.
First, break $\Cl(L) \otimes \Fp$ into its constituent pieces (as $\Gal(L/\QQ)$-modules).
More precisely, denote the $G_\QQ$-representation $\Sym^{i}(V)\otimes \Fp(j)$ by $A^{i,j}$ where $V\simeq \Fp^2$ with a prescribed Galois-action and write
\[
\Cl(L) \otimes \Fp \cong \bigoplus_{i,j} (A^{i,j})^{m_{ij}}
\]
We relate the number of the constituent pieces (with multiplicity) to the dimensions of the Selmer groups above.
In fact, we can write down a precise count for the multiplicities as well, namely

\begin{lthm}[Corollary~\ref{mij in terms of the rij's}]
Let $r_{ij}$ be defined in terms of the $\Fp$-dimensions of $\Lambda$-Selmer groups $H^1_{\Lambda}(G_\QQ, A^{i,j})$; see Corollary~\ref{define rij} for precise definitions.
For all $i,j$,
\[
m_{ij} = r_{ij} + r_{i, j+1} - r_{i-1, j+1} - r_{i+1, j}.
\]
\end{lthm}
This gives a precise description of the Galois module structure of $\Cl(L) \otimes \Fp$. 
From this precise Galois module structure, we conclude the $p$-rank of $L$ as desired.
\begin{lthm}[Theorem~\ref{main theorem of note 4}]
With notation as above,
\[
\rk_p(\Cl(L)) = \sum_{j=0}^{p-2}r_{p-1, j}.
\]
\end{lthm}

When $p$ is a regular prime, we obtain a slight refinement of the above equality; see the second assertion of Theorem~\ref{main theorem of note 4}.
As a consequence of the above results, we deduce a relation between the $\rk_p(\Cl(F))$ and $\rk_p(\Cl(L))$ when $p$ is regular. 

\begin{cor*}[Corollary~\ref{cor regular prime lower bound neat}]
Let $p$ be a regular prime\footnote{The lower bound can be made more robust when $p=3,5$.}.
Then
\[
\frac{p-7}{2} + 2\rk_p(\Cl(F)) \leq \rk_p(\Cl(L)).
\]    
\end{cor*}

\subsubsection*{\underline{\emph{Explicit upper and lower bounds}}}

The results discussed above are exact formulae, but the terms appearing in them are hard to compute since the Galois cohomology groups in question have coefficients in abstract Galois modules.
In Sections~\ref{sec: revised LB} and \ref{sec: revised UB}, we replace the abstract Selmer groups appearing in the formulae of $\rk_p(\Cl(L))$ (proved in Theorem~\ref{main theorem of note 4}) with more explicit ones, i.e., Galois cohomology groups with coefficients in $\Fp$ (or some twist); see also Remark~\ref{explicit count of dimensions}(i). The advantage of working with these  groups is that they are easier to compute; however, 
this replacement also comes at a cost: we no longer have precise formulae as before, instead we can only give lower and upper bounds. 
However, these bounds are significantly sharper than the bounds in Equation \eqref{bounds} obtained via classical class field theoretic methods; see Remark~\ref{explicit count of dimensions}(ii). The following is a quintessential example of the kind of results proven in Sections~\ref{sec: revised LB} and \ref{sec: revised UB}.

\begin{lthm} [Theorems~\ref{main result - revised lower bounds} and \ref{th: better upper bounds}] \label{thmE} 
Let $p$ be a regular prime.
Then
\[
\rk_p(\Cl(L)) \leq \frac{3p-5}{2} + (p-2)\sum_{i=2}^{p-2}\dim_{\Fp}\left( H^1_\Sigma(G_\QQ, \Fp(i))\right) + \sum_{\substack{i=2\\ \textrm{even}}}^{p-3}\dim_{\Fp}\left( H^1_\Lambda(G_\QQ, \Fp(i))\right).
\]
Here $ H^1_*(G_\QQ, \Fp(i))$ where $*\in \{ \Sigma, \Lambda\}$ are certain Selmer conditions.
On the other hand,
\[
\rk_p\left(\Cl(L)\right) \geq  \frac{p-1}{2} + \dim_{\Fp}\left( H^1_{\Sigma}(G_{\QQ}, \Fp(-1))\right) + \sum_{\substack{i=2 \\ i \text{ even}}}^{p-3}\dim_{\Fp}\left( H^1_{\Lambda}(G_{\QQ}, \Fp(i))\right).
\]
\end{lthm}

As mentioned above, the Selmer groups appearing in Theorem \ref{thmE} are easier to compute. 
In Theorem~\ref{last result} we provide a numerical criterion to determine when $\dim_{\Fp}\left( H^1_{\Lambda}(G_{\QQ}, \Fp(i))\right)=1$. 
This allows us to deduce a simple numerical criterion on upper and lower bounds for $\rk_p(\Cl(L))$. 

\begin{cor*}[Corollary \ref{lastcor}]
Let $p$ be a regular prime and  $i$ vary over even integers in the range $\{1, \ldots, p-2\}$.
Let $f$ be any element of order $p$ in $\mathbb{F}_N^\times$.
For an integer $0 < k < p-1$, define
\[
\mathcal{M}_k = (1-f)(1-f^2)^{2^{k}} \ldots (1-f^{p-1})^{(p-1)^{k}}.
\]
Then
\[
\frac{p-1}{2} + \alpha \leq \rk_p\left(\Cl(L)\right) \leq (p-1)(p-2) - (p-1)\left( \frac{p-1}{2}-1-\alpha\right)
,\]
where $\alpha$ is the number of $i \pmod{p-1}$ which are positive, even, and such that $\mathcal{M}_{p-1-i}$ is a $p$-th power in $\mathbb{F}_N^\times$. 
\end{cor*}

Since our corollary states that the size of $\rk_p(\Cl(L))$ depends on whether certain explicitly defined numbers are $p$-th powers in $\mathbb F_N^{\times}$, it can be regarded as an analogue of the results of Calegari--Emerton and Schaefer--Stubley mentioned at the start of the article. 


The results obtained by using tools from Galois cohomology are significantly stronger than those obtained via classical class field theoretic methods because the techniques in the latter case only consider the action of the Galois group $\Gal(L/\QQ(\zeta_p))$ on the class group of $L$, whereas techniques in the former case in addition also involve the Galois action by $\Gal(L/\QQ)$ on $\Cl(L)$.

\subsubsection*{\sc{Future Directions}}
An obvious problem to tackle in the future is to focus on the case $p=3$ and prove the heuristics when $N\equiv 1\pmod{9}$.
For this, we believe it would be required to develop new techniques which can detect how often ambiguous classes do not coincide with the strong ambiguous classes \emph{and} also qualitatively understand what determines growth in the $p$-rank of class group in going from a number field $F=\QQ(N^{1/p})$ to its Galois closure $L$.
Equivalently, it would be interesting to determine what `causes' the $3$-part of the class number to be at least 9.

We were able to do limited computations in the case $p=5$.
By results in \cite{SS19}, it is known that $\rk_5(\Cl(F))$ can take values 1,2, or 3.
Furthermore, it was predicted in \cite{SS19} that 
\[
\PP(\rk_5(\Cl(F))=1) = \left( 1 - \frac{1}{5}\right) = \frac{4}{5}.
\]
Our data suggests that the following prediction might be reasonable
\[
\PP(\rk_5(\Cl(L))=2) = \frac{2}{3}\left( 1 - \frac{1}{5}\right).
\]
When $\rk_5(\Cl(F)) = 1,2$, our data suggests that 
\[
\rk_5(\Cl(L)) \geq 2 \rk_{5}(\Cl(F)).
\]
On the other hand, when $\rk_5(\Cl(F)) = 3$, the dataset is small and we observe that  $\rk_5(\Cl(L)) \geq 8$.
Theoretically the best result we have proven in this direction is Corollary~\ref{p=5 cor for lb} but we do not know whether the theoretical bounds are in fact attainable.
In the same vein, our data for $p=5$ suggests that $\rk_5(\Cl(L))$ is at most 10 (even though a priori the 5-rank may be 11 or 12).
Can our methods be refined to obtain even sharper upper bounds?
Further calculations for $p\geq 7$ will assist in making a general prediction for the distribution of $p$-ranks of the class group of $L$ and also determine whether our (theoretical) results are sharp. 

Finally, it would be interesting to study the quantity $\alpha$ and the numbers $\mathcal{M}_k$ in greater detail.
It was pointed out to us by experts that it might be reasonable to expect that these quantities are related to zeta values.
In the future, it would be interesting to study the distribution of $\log(\mathcal{M}_k)$.

\subsubsection*{\sc{Organization}}
Including the introduction, this article has \emph{four} sections.
Section~\ref{sec: prelim} is preliminary in nature.
We record useful results and provide proofs of basic facts that will be required throughout the paper.
Section~\ref{sec: p rank via cft} is aimed at studying $p$-rank of class groups via class field theory.
In particular, we obtain possible upper and lower bounds using tools like the Chevalley class number formula in Section~\ref{sec: UB and LB}.
In the case that $p=3$, we prove the distribution of the $p$-rank of $\QQ(\zeta_3, N^{1/3})$ where $N\equiv 4,7\pmod{9}$ in Section~\ref{sec: N 4,7 mod 9}.
To handle the case of $N\equiv 1\pmod{9}$, we appeal to results of F.~Gerth and the notion of (strictly) ambiguous classes.
In this case, we can not obtain precise theorems but are able to provide heuristics which support the data.
In Section~\ref{sec: p rank via Gal cohom} we study $p$-ranks of class groups via Galois cohomology.
The main goal of this section is provide refinements of results proven via class field theory.
In particular, when $p$ is a regular prime we provide sharper upper and lower bounds than previously recorded in literature.
The cohomological theory developed in this section is technical and builds on the work of \cite{SS19}.

\subsubsection*{\sc{Acknowledgements}}
This collaboration started at Rethinking Number Theory IV (June 2023).
We thank the organizers for giving us the opportunity and space to collaborate.
The authors are grateful to Brandon Alberts, Rahul Arora, Carlo Pagano, Karl Schaefer, Arul Shankar, Jiuya Wang, and Stanley Xiao for answering many questions during the preparation of this article, and would also like to thank Robin Visser for help with generating computational data. 
We thank Preston Wake for his comments on an earlier version of the preprint.

We acknowledge the financial support (for technology) provided by NSF via grant DMS-2201085 and by UW-Eau Claire for running RNT IV.
AS is supported by funding from the European Research Council under the European Union’s Horizon 2020 research and innovation programme (Grant agreement No. 101001051 — Shimura varieties and the Birch--Swinnerton-Dyer conjecture).

\section{Preliminaries}
\label{sec: prelim}
Let $p$ be an odd prime and $N$ be a (different) prime such that $N\equiv 1\pmod{p}$.
Write $\zeta_p$ to denote a primitive $p$-th root of unity.
Set $K=\mathbb{Q}(\zeta_p)$, $L=\mathbb{Q}(\zeta_p, N^{1/p})$, and write $G=\Gal(L/K)$.
As is standard, write $\cO_K$ (resp. $\cO_L$) to denote the ring of integers of $K$ (resp. $L$).

Recall the following fact which will be required in subsequent proofs.

\begin{Th}
\label{Gras}
Set $K= \QQ(\zeta_p)$ and $L=K(N^{1/p})$.
For a finite place $v$ in $K$ write $\fp_v$ to denote the corresponding prime ideal and $\iota_v$ to denote the embedding of $K\hookrightarrow K_v$.
\begin{enumerate}
\item[\textup{(i)}] Let $v\nmid p$.
The place $v$ is ramified in $L/K$ if and only if the normalized ($\fp_v$-adic) valuation $\val(N) \not\equiv 0 \pmod{p}$.
If $v$ is unramified in $L/K$, it is split if and only if $\iota_v(N)\in K_v^{\times p}$.

\item[\textup{(ii)}] Let $v\mid p$.
Then $v$ is ramified is $L/K$ if and only if there does not exist $x\in K^\times$ satisfying the following congruence relation
\[
\frac{N}{x^p} \equiv 1 \pmod{\fp_v^p}.
\]
Furthermore, if $v$ is unramified in $L/K$ then it is split if and only if $\iota_v(N)\in K_v^{\times p}$.
\end{enumerate}

\end{Th}

\begin{proof}
This is a specific case of \cite[Chapter I, Theorem~6.3]{Gras_CFT}.
\end{proof}

\begin{Lemma}
\label{p-ramification}
The prime $\pi=(1-\zeta_p)$ in $\cO_K$ ramifies in $L/K$ if and only if $N\not\equiv1\pmod{p^2}$.
Furthermore, if $\pi$ does not ramify then it splits.
\end{Lemma}

\begin{proof}

\underline{$\pi$ ramifies $\Rightarrow N\not\equiv1\pmod{p^2}$.}

First consider the case that $N\equiv 1\pmod{p^2}$.
Note that
\[
\frac{N}{1^p}\equiv 1\pmod{\pi^p}.
\]
Theorem~\ref{Gras}(ii) asserts that $\pi$ is not ramified.
Write $K_p$ to denote the completion of $K$ under $\pi-$adic topology and set $U^{(k)}=1+\pi^k\cO_{K_p}$.
In the following commutative diagram, the horizontal arrows are isomorphisms and the vertical arrows are inclusions
\[
\begin{tikzcd}
    U^{(p-1)} \arrow[hookleftarrow]{d} & \pi^{p-1}=p\cO_{K_p} \arrow[swap]{l}{\exp_p}\arrow[hookleftarrow]{d} \\
    U^{(2p-2)}\arrow{r}{\log_p} & \pi^{2p-2}=p^2\cO_{K_p}.
\end{tikzcd}
\]
Since $(p^2)=\pi^{2p-2}$, it follows that $N\in U^{(2p-2)}$ and $\log_p(N)=py$ for some $y\in p\cO_K$.
Hence, $N=\exp_p(y)^p\in K_p^{\times p}$.
Therefore, $\pi$ splits completely in $L/K$.
\newline

\underline{$N\not\equiv 1\pmod{p^2} \Rightarrow \pi$ ramifies.}

Now, suppose that $N\not\equiv1\pmod{p^2}$ and write $N=1+p\alpha$ with $\alpha\in \ZZ_{>0}$ such that $\gcd(\alpha,p)=1$.
\newline

\emph{Claim:} There does not exist any $x\in K_p^{\times}$ such that
\[
\frac{N}{x^p}\equiv 1\pmod{\pi^p}.
\]

\emph{Justification:}
Suppose, there exists such an element $x$.
Then $x\in \cO_{K_p}^{\times}$ and $x^p \equiv N\pmod{\pi^p}$.
In other words, for some $\beta, \beta' \in \cO_{K_p}$ 
\[
x^p=N+(1-\zeta_p)^p\beta=1+p(\alpha+\beta'(1-\zeta_p))\in U^{(1)}\setminus U^{(p)}.
\]
Now, as $x\in \cO_{K_p}^{\times}\cong \mathbb{F}_p^{\times} \times U^{(1)}$, and since, taking the $p$-th power doesn't change the first component, it follows that $x\in U^{(1)}$.
Therefore, $x^p\in (U^{(1)})^p=(1+\pi \cO_{K_p})^p\subseteq U^{(p)}$ which is a contradiction.
\newline

Theorem~\ref{Gras}(ii) now implies that $\pi$ ramifies in $L/K$.
This completes the proof.\end{proof}

\begin{Lemma} \label{N-ramification}
Keep the notation introduced before.
Then $N$ splits completely in the cyclotomic extension $K/\mathbb{Q}$ and the primes above $N$ are totally ramified in $L/K$.
Moreover, the extension $L/K$ is unramified away from $N$ and $p$.
\end{Lemma}

\begin{proof}
Assuming that $N\equiv 1 \pmod{p}$ implies that $\zeta_p\in \mathbb{Q}_N$ and that $N$ splits completely in $K/\QQ$.
Since, $N^{1/p}\in L$, the primes above $N$ (in $K$) ramify in the extension $L/K$.
The second claim follows from Theorem~\ref{Gras}(i).
\end{proof}

Let us recall the notion of the \textit{group of id\'eles} at $L$, denoted by $I_L$; 
\[
I_L=\prod_{\fQ}' L_{\fQ}=\prod_{\fq}' \prod_{\fQ \mid \fq} L_{\fQ}
\]
where the first equality is the restricted product running through all places $\fQ$ in $L$.
In the second equality, the restricted product runs through all places $\fq$ in $K$.

In the following result we provide a characterization of when $\zeta_p$ is a norm element.

\begin{Th}
\label{Norm}
$N\equiv 1 \pmod{p^2}$ if and only if $\zeta_p\in \Norm_{L/K}(L^{\times})$.
\end{Th}

\begin{proof}
\uline{$N\equiv 1 \pmod{p^2} \Rightarrow \zeta_p\in \Norm_{L/K}(L^{\times})$.}
Since $L/K$ is cyclic, by Haase norm theorem \cite[Theorem V.4.5]{Jan73} , it is enough to show that $\zeta_p\in \Norm_{L/K}(I_L)$.
For $\alpha =(\alpha_{\fQ})_{\fQ}\in I_L$, the norm map $\Norm_{L/K}(\alpha)$ can be defined by,
\[
(\Norm_{L/K}(\alpha))_{\fq}=\prod_{\fQ \mid \fq}\Norm_{K_{\fq}}^{L_\fQ}(\alpha_{\fQ})=\prod_{\fQ \mid \fq} \prod_{\sigma\in G_{\fQ}}\sigma(\alpha_{\fQ}), 
\]
where $G_\mathfrak Q$ is the Galois group of the extension $L_{\fQ}/K_{\fp}$.

Our goal is to find $\alpha =(\alpha_{\fQ})_{\fQ}\in I_L$ such that $\zeta_p=\Norm_{L/K}(\alpha)$; here $\zeta_p$ is viewed as an inside $I_L$ via the diagonal map.
Now, we give $\alpha_{\fQ}$ explicitly at each place $\fQ$.

Suppose that $\fq$ splits in $L/K$ into $\fQ_1, \fQ_2, \ldots, \fQ_p$.
Set $\alpha_{\fQ_1}=\zeta_p$ and $\alpha_{\fQ_i}=1$ when $i\neq 1$.
So, $(\Norm_{L/K}(\alpha))_{\fq}=\zeta_p\cdot 1 \cdot \ldots 1=\zeta_p$ as desired.

If $\fq$ ramifies in $L/K$ by Lemmas~\ref{p-ramification} and~\ref{N-ramification}, it follows from our assumption on the congruence condition on $N$ that $\fq$ is a prime above $N$.
Moreover, the condition on $N$ also implies that $\zeta_{p^2}\in \mathbb{Q}_N=K_{\fq}$.
Set $\alpha_{\fQ}=\zeta_{p^2}$.
Then, $(\Norm_{L/K}(\alpha))_{\fq}=\zeta_{p^2}^p=\zeta_{p}$ as desired.

Finally, suppose that $\fq$ is inert in $L/K$.
Lemma~\ref{p-ramification} asserts that $\fq\nmid p$.
Therefore $\zeta_p\in \mathbb{F}_{\fq}$, the residue field at $\fq$.
Now, if $\mathbb{F}_{\fQ}$ is the residue field at $\fQ$ in $L$, then $\mathbb{F}_{\fQ}/\mathbb{F}_{\fq}$ is a degree $p$ extension.
Since the norm map between finite fields is surjective, there exist an element $x\in \mathbb{F}_{\fQ}$ such that $\zeta_p=\Norm_{\mathbb{F}_{\fQ}}^{\mathbb{F}_{\fq}}(x)= \Norm_{K_{\fq}}^{L_\fQ}(x)$.
Set $\alpha_{\fQ}=x$ to obtain the desired result.
Here, we are viewing $ \mathbb{F}_{\fQ}$ inside $L_{\fQ}$ via the Teichm{\"u}ller lift.
\newline

\uline{$\zeta_p\in \Norm_{L/K}(L^{\times}) \Rightarrow N\equiv 1 \pmod{p^2}$.}
Since $N$ ramifies in $L/K$, we have $\zeta_p\in \Norm_{K_N}^{L_N}(L_N^{\times})$ where $L_N$ and $K_N$ are the completions of $L$ and $K$ respectively under $N$-adic topology. 
Since, $\zeta_p\in \mathbb{Q}_N$, it follows that $K_N=\mathbb{Q}_N$ and $L_N=\mathbb{Q}_N(N^{1/p})$.
By assumption, there exists $x\in L_N^{\times}$ such that $\zeta_p=\Norm_{K_N}^{L_N}(x)$.
So,
\[
x\in \cO_{L_N}^{\times}\cong \mathbb{F}_N^{\times}\times U_{L_N}^{(1)}
\]
where $U_{L_N}^{(1)}$ is the group of 1-units of $L_N$.
Via the above isomorphism, $x=\zeta\cdot u$ for some $\zeta\in \mathbb{F}_N^{\times}$ and $u\in U_{L_N}^{(1)}$.
Note that $G_N=\Gal(L_N/K_N)$ acts on $U_{L_N}^{(1)}$ in the obvious way and on $\mathbb{F}_N^{\times}$, trivially.
Therefore,
\[
\zeta_p=\Norm_{L_N/K_N}(x)=\prod_{\sigma\in G_N}\sigma (\zeta\cdot u)=\zeta^p\prod_{\sigma\in G_N}\sigma (u)\in \mathbb{F}_N^{\times}\times U_{L_N}^{(1)}.
\]
This forces $\zeta^p=\zeta_p$.
So, $\zeta$ is a primitive $p^2$-th root of unity in $\mathbb{F}_N^{\times}$ and $N\equiv 1 \pmod{p^2}$.
\end{proof}

\section{\texorpdfstring{$p$-rank of the class group of $L$ via class field theory}{}}
\label{sec: p rank via cft}

This section is devoted to study the $p$-rank of the class group of the number field $L = \QQ(\zeta_p, N^{1/p})$ where $N\equiv 1 \pmod{p}$ is a prime.
Henceforth, we write the class group of $L$ by $\Cl(L)$.

\begin{definition}
For any (finite) abelian group $A$, the \emph{$p$-rank of $A$} is defined to be 
\[
\rk_p(A):=\dim_{\mathbb{F}_p}(A\otimes \mathbb{F}_p).
\]
Equivalently,
\[
\rk_p(A):=\dim_{\mathbb{F}_p}A[p] = \dim_{\mathbb{F}_p}A/pA.
\]
\end{definition}

We use elementary methods to show (na{\"i}ve) upper and lower bounds on the $p$-rank of $\Cl(L)$.
When $p$ is a regular prime, better estimates are possible but it is not clear from our computer experiments whether the (upper) bounds are sharp.
For example, when $p=5$ theoretical calculations show that the $5$-rank $\Cl(L)$ is at most 12; but varying over primes $N\equiv 1\pmod{5}$ we have only obtained a maximum of $5$-rank equal to 10 in our computer experiments.
In the special case of $p=3$, we prove that $\rk_3(\Cl(3))$ is either 1 or 2; our computer experiments show that both these values are indeed attained.
Our data{\footnote{code available upon request}} suggests that varying over primes of the form $N\equiv 1\pmod{3}$, the distribution of $\rk_3(\Cl(L)) =1$ (resp. 2) is 2/3 (resp. 1/3).
We provide a partial proof towards this observation and also provide heuristic arguments for explaining the data.

\subsection{Upper and lower bounds for \texorpdfstring{$\rk_p\left(\Cl(L)\right)$}{} }
\label{sec: UB and LB}
Set $K = \QQ(\zeta_p)$ and $F=\QQ(N^{1/p})$.
By genus theory, it is possible to show that the degree-$p$ subfield of $F(\zeta_N)/F$ is unramified everywhere.
Therefore it follows that $\rk_p(F)\geq 1$.
Since the $p$-part of the class group of $F$ injects into the $p$-part of the class group of $L$, it follows that $\rk_p\left(\Cl(F)\right)\geq 1$.

Our first order of business is to find an \emph{optimal} lower bound for $\rk_p\left(\Cl(L)\right)$.
The first result is to provide a lower bound for the $p$-rank of the class group of $L$.

\begin{Th}
With notation as above, $\rk_p\left(\Cl(L)\right)\geq \frac{p-1}{2}$. 
\label{lower bound theorem}
\end{Th}

\begin{proof}
Let $t$ denote the number of finite primes that ramify in $L/K$ and $r_1^c$ be the number of real places of $K$ that get complexified in $L/K$.
\cite[Chapter IV, Corollary 4.5.1]{Gras_CFT} asserts that
\begin{equation}
\label{gras lower bound}
\rk_p\left(\Cl(L)\right)\geq t+r_1^c-1-\rk_p(\cO_K^{\times}/\cO_K^{\times}\cap \Norm_{L/K}(I_L) )
\end{equation}
Note that $K$ is totally complex, so $r_1^c=0$ in our case.
\newline

\underline{Case 1: when $N\equiv 1\pmod{p^2}$.}  \newline

Lemmas~\ref{p-ramification} and \ref{N-ramification} imply that $t=p-1$.
By Dirichlet unit theorem $\cO_K^{\times}\cong \langle \zeta_{2p} \rangle \times \ZZ^r$ where $r=\frac{p-1}{2}-1$.
Also $\cO_K^{\times p}\subseteq \Norm_{L/K}(I_L)$. 
Together with Theorem \ref{Norm}
\[
\langle \zeta_p, \cO_K^{\times p} \rangle\subseteq \cO_K^{\times}\cap \Norm_{L/K}(I_L).
\]
Via the isomorphism coming from the Dirichlet unit theorem, we have $\langle \zeta_p, \cO_K^{\times p} \rangle \cong \langle \zeta_{2p} \rangle \times (p\ZZ)^r$.
Moreover,
\[
(\ZZ/p\ZZ)^r\xrightarrow[]{\sim} \cO_K^{\times}/ \langle \zeta_{p}, \cO_K^{\times p} \rangle \twoheadrightarrow{} \cO_K^{\times}/\cO_K^{\times}\cap \Norm_{L/K}(I_L).
\]
Therefore, $\rk_p(\cO_K^{\times}/\cO_K^{\times}\cap \Norm_{L/K}(I_L))\leq r$.
Hence, \eqref{gras lower bound} simplifies in this case to yield
\[
\rk_p\left(\Cl(L)\right)\geq (p-1)+0-1-\left(\frac{p-1}{2}-1\right)=\frac{p-1}{2}.
\]

\underline{Case 2: when $N\not\equiv 1\pmod{p^2}$.} \newline

Applying Lemmas~\ref{p-ramification} and ~\ref{N-ramification} shows that $t=(p-1)+1=p$.
It is trivially true that
\[
\rk_p(\cO_K^{\times}/\cO_K^{\times}\cap \Norm_{L/K}(I_L))\leq \rk_p(\cO_K^{\times}) = r+1=\frac{p-1}{2}.
\]
Therefore, \eqref{gras lower bound} simplifies in this case to yield
\[
\rk_p\left(\Cl(L)\right)\geq p+0-1-\frac{p-1}{2}=\frac{p-1}{2}. \qedhere
\] 
\end{proof}

We now work towards proving an upper bound of the $p$-rank of $\Cl(L)$ in terms of the $p$-rank of $\Cl(K)$.
We prove some lemmas on the $p$-rank of the $G$-invariance of the class group of $L$.

\begin{Lemma}
\label{G-invariance upper bound}
As before, write $G=\Gal(L/K)$.
Then,
\[
\rk_p(\Cl(L)^G)\leq \rk_p(\Cl(K))+\frac{3}{2}(p-1).
\]
\end{Lemma}

\begin{proof}
Set $s$ to denote the quantity $\rk_p(\cO_K^{\times}\cap \Norm_{L/K}(L^{\times})/\Norm_{L/K}(\cO_L^{\times}))$ and write $t$ to denote the number of finite places that ramify in $L/K$.
\newline

\emph{Claim:} With notation as above, $s+t\leq\frac{3}{2}(p-1)$.
\newline

\emph{Justification:}
We prove the claim case-by case. \newline

\underline{Case 1: When $N\not\equiv 1 \pmod{p^2}$.} \newline

In this case, $t=(p-1)+1=p$.
Recall from Theorem~\ref{Norm} that $\zeta_p\not\in \Norm_{L/K}(L^{\times})$; so, writing $F(\cO_K^{\times})$ to denote the free part of $\cO_K^{\times}$,
\[
F(\cO_K^{\times})\twoheadrightarrow F(\cO_K^{\times})/(\cO_K^{\times})^p\supseteq \cO_K^{\times}\cap \Norm_{L/K}(L^{\times})/(\cO_K^{\times})^p \twoheadrightarrow \cO_K^{\times}\cap \Norm_{L/K}(L^{\times})/\Norm_{L/K}(\cO_L^{\times}).
\]
It follows that, $s\leq \rk_{\ZZ}(F(\cO_K^{\times}))=\frac{p-1}{2}-1$.
Therefore,
\[
s+t\leq p+\frac{p-1}{2}-1=\frac{3}{2}(p-1).
\]

\underline{\emph{Case 2: When $N\equiv 1 \pmod{p^2}$}.} \newline

In this case, note that $t=p-1$. 
Also,
\[
s\leq \rk_{p}(\cO_K^{\times})=1+\frac{p-1}{2}-1=\frac{p-1}{2}.
\]
Hence, $s+t\leq \frac{3}{2}(p-1)$.
This completes the proof of the claim.
\newline

An application of \cite[Proposition~2.4]{Gras_22_paper} implies that
\[
\rk_p(\Cl(L)^G)\leq \rk_p(\Cl(K))+t+s\leq \rk_p(\Cl(K))+\frac{3}{2}(p-1). \qedhere
\]
\end{proof}

\subsection{The case of regular prime}
When $p$ is a regular prime, better estimates can be obtained.

\begin{definition}
A prime $p$ is called \emph{regular} if $p$ does not divide the size of the class group of $\QQ(\zeta_p)$. 
\end{definition}

\begin{Lemma}
\label{G-invariance upper bound regular}
If $p$ is a regular prime, then \[
\rk_p(\Cl(L)^G)\leq p-2.
\]
\end{Lemma}

\begin{proof}
We prove the result by showing that 
\[
\abs{\Cl(L)^G\otimes\Fp}\leq p^{p-2}.
\]
As in the previous lemma, set $t$ to denote the number of finite primes that ramify in $L/K$.
Recall Chevalley's ambiguous class number formula (see for example \cite[Chapter~II, \S6.2.3]{Gras_CFT}) which asserts that when $p$ is a regular prime,
\[
\abs{\Cl(L)^G\otimes\Fp}=\frac{p^{t-1}}{p\text{-part of }[\cO_K^{\times}:\cO_K^{\times}\cap \Norm_{L/K}(L^{\times})]}.
\]

\underline{Case 1: When $N\not\equiv 1 \pmod{p^2}$.} \newline

As before $t=(p-1)+1=p$.
Since $\zeta_p\in \cO_K^{\times}\setminus \Norm_{L/K}(L^{\times})$, it follows from Theorem \ref{Norm} that $p\mid [\cO_K^{\times}:\cO_K^{\times}\cap \Norm_{L/K}(L^{\times})]$.
Hence,
\[
\abs{\Cl(L)^G\otimes\Fp}\leq \frac{p^{p-1}}{p}=p^{p-2}.
\]

\underline{Case 2: When $N\equiv 1 \pmod{p^2}$.} \newline

In this situation $t=p-1$ and 
\[
\abs{\Cl(L)^G\otimes\Fp}\leq p^{t-1}=p^{p-2}. \qedhere
\]
\end{proof}

We can now state and prove the main result of this section.

\begin{Th}
\label{cft ub theorem}
With notation as above,
\[
\rk_p\left(\Cl(L)\right)\leq p \rk_p(\Cl(K))+\frac{3}{2}(p-1)^2.
\]
Furthermore, if $p$ is a regular prime, then 
\[
\rk_p\left(\Cl(L)\right)\leq (p-1)(p-2).
\]
\end{Th}
\begin{proof}
It follows from \cite[(1) and Corollary~2.3]{Gras_22_paper} that
\[
\rk_p\left(\Cl(L)\right)\leq \rk_p(\Cl(K))+(p-1)\rk_p(\Cl(L)^G).
\]
By Lemma~\ref{G-invariance upper bound} it follows that
\[
\rk_p\left(\Cl(L)\right)\leq \rk_p(\Cl(K))+(p-1)\left(\rk_p(\Cl(K))+\frac{3}{2}(p-1)\right)=p\rk_p(\Cl(K))+\frac{3}{2}(p-1)^2.
\]

Furthermore, if $p$ is regular then $\rk_p(\Cl(K))=0$.
Lemma~\ref{G-invariance upper bound regular} then implies
\[
\rk_p\left(\Cl(L)\right)\leq (p-1)(p-2).\qedhere
\]
\end{proof}

\subsection{The case when \texorpdfstring{$p=3$}{}}
\label{sec: p=3 cft}

Set $\pi=1-\zeta_3$.
Any non-zero element of $\ZZ[\zeta_3]$ can be written as $\pm \zeta_3^{i_1}\pi^{i_2}u$, where $i_1, i_2$ are non-negative integers and $u\equiv 1 \pmod{3}$.
Since $N\equiv 1 \pmod{3}$, it splits as $N= \fn \overline{\fn}$ in $\ZZ[\zeta_3]$; we often write $\fn = \fn_1$ and $\overline{\fn} = \fn_2$.

\begin{Lemma}
\label{4n as a sum}
When $N\equiv 1 \pmod{3}$, for some of choice of integers $A,B$ (unique up to sign), 
\[4N= A^2+27B^2.\]
\end{Lemma}

\begin{proof}
First factor the above equation in $\ZZ[\zeta_3]$ into
\[
\frac{A^2+27B^2}{4}=\left(\frac{A+3B\sqrt{-3}}{2}\right) \left(\frac{A-3B\sqrt{-3}}{2}\right).
\]
Writing $N=\fn\overline{\fn}$, suppose that $\fn= a + b\zeta_3$.
Then we may rewrite $\fn$ as follows
\[
\fn=\frac{2a-b+b\sqrt{-3}}{2}.
\]
This equation is unique up to multiplication of $\zeta_3$ and $\zeta_3^2$.
In other words,
\begin{align*}
\fn &=\frac{-(a+b)+(a-b)\sqrt{-3}}{2}\\
\fn &=\frac{2b-3a+a\sqrt{-3}}{2}
\end{align*}
We want $a$, $b$, or $a-b\equiv0\pmod{3}$. 
Factorization of $N$ implies that
\[
N=(a+b\zeta_3)(a+b\zeta_3^2)=a^2-ab+b^2 \equiv 1 \pmod{3}.
\]
Therefore, either $a$ or $b\equiv0\pmod3$ or $a\equiv b\pmod 3$.
Note that the `or's are exclusive, which gives us the uniqueness up to sign.  

In particular, we choose $A = (2a-b)$ and $B = \frac{b}{3}$.
\end{proof}

Results in the previous section assert that $1\leq \rk_3 \Cl(L)\leq 2$.
The main result of this section is to provide a precise characterization of $\rk_3\Cl(L)$.

\begin{rem}
Note that the absolute discriminant of $L/\QQ$ is $-3^7 N^4$ when $N\not\equiv 1 \pmod{9}$, whereas it is $-3^3 N^4$ when $N\equiv 1 \pmod{9}$.
In the former case, $3$ is totally ramified in $L$ whereas in the latter situation a prime above $3$ has ramification index 2; see Theorem~\ref{p-ramification}. 
\end{rem}

\subsubsection{\textbf{Hilbert Symbol calculations and applications}}
We first prove the theorems regarding the cubic Hilbert symbols attached to the extension $L/K$.
For the definition and basic properties see \cite[Chapter V, Proposition~3.2]{Neu99}.

\begin{Lemma}
\label{Hilbert-pi}
If $x=\fn_1$ or $\fn_1^2\fn_2$ and $j=1$ or $2$, then $\left(\frac{x,N}{(\fn_j)}\right)_3=1$. 
\end{Lemma}

\begin{proof}
By \cite[Chapter V, Proposition~3.4]{Neu99}, 
\[
\left(\frac{\fn_1,\fn_1}{(\fn_2)}\right)_3=\left(\frac{\fn_2,\fn_2}{(\fn_1)}\right)_3=1.
\]
Same proposition implies that 
\[
\left(\frac{\fn_1,\fn_1}{(\fn_1)}\right)_3 = \left(\frac{\fn_2,\fn_2}{(\fn_2)}\right)_3 = (-1)^{\frac{N-1}{3}}=1
\]
since $N$ is an odd prime.
On the other hand, \cite[Chapter V, Proposition~3.5]{Neu99} asserts that 
\[
\left(\frac{\fn_1,\fn_2}{(\fn_1)}\right)_3=1 \Longleftrightarrow \fn_2 \textrm{ is a cube}\pmod{\fn_1}.
\]
Since $\fn_2=(2a-b)-\fn_1$,
\[
\left(\frac{\fn_1,\fn_2}{(\fn_1)}\right)_3 = 1 \Longleftrightarrow 2a-b \textrm{ is a cube}\pmod{\fn_1},
\]
Equivalently, $A=2a-b$ is a cube $\pmod{N}$.
By \cite[Corollary~7.6]{Lem_RecLaw} and Wilson's theorem if $N=3m+1$, we have $A\cdot (m!)^3 \equiv 1\pmod{N}$.
So $A$ is always a cube modulo $N$.
Hence,
\[
\left(\frac{\fn_1,\fn_2}{(\fn_1)}\right)_3 = 1.
\]
A similar result holds for $\left(\frac{\fn_1,\fn_2}{(\fn_2)}\right)_3$.
For each $x$ and $j$, the claim follows from basic properties of Hilbert symbols.
\end{proof}

\begin{Lemma}
\label{Hilbert-lambda}
Suppose that $x=\fn_1^2\fn_2\equiv 1\pmod{3\pi}$.
Then, 
\[
\left(\frac{x,\pi}{(\pi)}\right)_3=1 \Longleftrightarrow 3\mid B.
\]
\end{Lemma}

\begin{proof}
Our assumption $x\equiv 1\pmod{3\pi}$ implies that $\log(x) \equiv x-1\pmod{9\pi}$. 
Here, we are using $3$-adic logarithm.
Consider the trace map 
\[
\tau: \mathbb{Q}_p(\zeta_3)\longrightarrow \mathbb{Q}_p;
\]
note that $\frac{\zeta_3}{\pi}(x-1)$ maps to $1-Na$ under $\tau$.
This together with the previous congruence implies, 
\[
\tau\left(\frac{\zeta_3}{\pi}\log(x)\right)\equiv 1-Na \pmod{9}.
\]
Now, in view of \cite[(3) on p.~340]{Neu99} and the above congruence,
\[
\left(\frac{x,\pi}{(\pi)}\right)_3 = 1 \Longleftrightarrow Na\equiv 1 \pmod{9}.
\]
Recall (from Lemma~\ref{4n as a sum}) that $N=\fn_1\fn_2 = a^2-ab+b^2 \equiv 1 \pmod{3}$. 
Since, $\fn_i\equiv 1 \pmod{3}$ we have $a\equiv 1,4 \text{ or } 7 \pmod{9}$ and $3\mid b$. Therefore, $Na-1\equiv a^2(a-b)-1\equiv -a^2b\pmod{9}$.
Hence,
\[
Na\equiv 1 \pmod{9} \Longleftrightarrow 9\mid b
\]
and this completes the proof. 
\end{proof}

In the following theorem we provide a characterization of when $\rk_3\left(\Cl(L)\right)=1$ or $2$.
We begin by recording a result of F.~Gerth in our particular setting; see \cite[Theorem~5.3]{Ger76}.

\begin{theorem*}[Gerth]
Let $K=\QQ(\zeta_3) = \QQ(\sqrt{-3})$ and $L = K(\sqrt[3]N)$ where $N = \fn_1 \fn_2 \equiv 1 \pmod{3}$ and each\footnote{In Gerth's notation $g=2$ for our setting.} $\fn_i \equiv 1\pmod{\pi^3 \cO_K}$.
Let $M_1 = L(\sqrt[3]{x_1}) = L(\sqrt[3]{\fn_1 \fn_2^2}) = L(\sqrt[3]{\fn_2 \fn_1^2})$ denote the genus field\footnote{In Gerth's notation $t=1$ for our setting by \cite[Theorem~5.1]{Ger76} and the choice of $x_1$ follows from \cite[Theorem~5.2]{Ger76}.}.
If there exist ambiguous ideal classes of $L/K$ which are not strong ambiguous, let $\fP$ be a prime ideal of $L$ contained in one such class which is relatively prime to $x_1$.
Let $\upsilon$ be a prime element of $F$ such that $(\upsilon) = \Norm_{L/K}(\fP)$.
Let $s$ denote the rank of the matrix $(\beta_{1j})$ where $1\leq j \leq u$ and each $\beta_{1j}\in \mathbb{F}_3$.
Here, $u$ and $\beta_{1j}$ are defined as follows:

\begin{align*}
u & = \begin{cases}
  2 & \textrm{ if } (\pi) \textrm{ does not ramify in }L/K \textrm{ and all ambiguous classes are strong ambiguous}\\
  4 & \textrm{ if } (\pi) \textrm{ ramifies in }L/K \textrm{ and there exist ambiguous classes which are not strong ambiguous}\\
  3 & \textrm{ otherwise.}
\end{cases} \\
\zeta_3^{\beta_{1j}} & = \begin{cases}
\left( \frac{x_1, N}{(\fn_j)}\right)_3 & \text{ 
if } 1\leq j \leq 2 \\
\left( \frac{x_1, \pi}{(\pi)}\right)_3 & \text{ 
if } j =3 \text{ and }(\pi) \text{ ramifies in }L/K \\
\left( \frac{x_1, \upsilon}{(\upsilon)}\right)_3 & \text{ 
if } j=u \text{ and if there exist ambiguous ideal classes which are not strong ambiguous}
\end{cases}.
\end{align*}
Then $\rk_3\left(\Cl(L)\right) = 2-s$.
\end{theorem*}

\begin{Th}
Keep the notation introduced above.
\begin{enumerate}[\textup{(}i\textup{)}]
\item If $N\not\equiv 1 \pmod{9}$, then $\rk_3\left(\Cl(L)\right)=2$ if and only if $3\mid B$.
\item If $N \equiv 1 \pmod{9}$, then $\rk_3\left(\Cl(L)\right)=2$ if and only if $A$ is a 9th power modulo $N$.
\end{enumerate}
\label{criterion}
\end{Th}

\begin{proof}

\begin{enumerate}[\textup{(}i\textup{)}]
\item 
Since $N\not\equiv 1 \pmod{9}$ by assumption, we have $\fn_i\equiv4 \text{ or }7 \pmod{3\pi}$.
By \cite[Remark on p.~98]{Ger76}, there are no ambiguous ideal classes which are not strongly ambiguous in $L/K$.
So, $u=3$.
In our setting $t=1$ and that we can choose $x_1=\fn_1^2\fn_2$.

Gerth's theorem (above) asserts that 
\[
s=0 \Longleftrightarrow \left(\frac{\fn_1^2\fn_2, N}{\fn_i}\right)=\left(\frac{\fn_1^2\fn_2,\pi}{(\pi)}\right)=1.
\]
By Lemmas~\ref{Hilbert-pi} and~\ref{Hilbert-lambda}, this condition is satisfied precisely when $3\mid B$.
The proof follows from the observation that $\rk_3\left(\Cl(L)\right)=2-s$.
\item As mentioned in the proof of Lemma~\ref{Hilbert-pi}, we have $A\cdot ((\frac{N-1}{3})!)^3 \equiv 1 \pmod{N}$. This together with \cite[Theorem 1.3]{Calegari-Emerton} implies the claim. \qedhere
\end{enumerate}
\end{proof}

In the next two sections, we prove distribution results.
We show `how often' $\rk_3\left(\Cl(L)\right)$ takes the values $1$, $2$ as $N$ varies over all primes $N\equiv 1\pmod{3}$.

\subsubsection{\textbf{The case when \texorpdfstring{$N\not\equiv 1\pmod{9}$}{}}}
\label{sec: N 4,7 mod 9}

The main goal of this section is to analyze the case when $N\equiv 4,7 \pmod{9}$.
More precisely, we show that
\[
\PP\left(\rk_3\left(\Cl(L)\right)=2 \textrm{ \& } N\equiv 4,7 \pmod{9}\right) = \frac{1}{3} \textrm{ and } \PP\left(\rk_3\left(\Cl(L)\right)=1 \textrm{ \& } N\equiv 4,7 \pmod{9}\right) = \frac{2}{3}.
\]

\subsubsection*{\underline{Recollections: Ray Class Groups}}
For this section, we write $K$ to denote any (generic) number field.
Let $\fm$ be an ideal of $\cO_{K}$ and $I(K)$ the set of ideals in $\cO_{K}$.
Define 
\[
I_\fm(K)=\left\{\fa\in I(K):\fa+\fm=\cO_{K}\right\}
\]
to be the set of ideals in $\cO_{K}$ coprime to $\fm$.
Let $P(K)$ be the principal ideals of $\cO_{K}$.
Next, define 
\[
P_{\fm,1}(K)=\left\{(\alpha)\in P(K) : 
\alpha\equiv1\pmod{\fm}\right\}.
\]
For $I$ and $J\in I_\fm(K)$ we say $I\sim J$ if there exist ideals $(\alpha)$ and $(\beta)\in P_{\fm,1}(K)$ such that $I(\alpha)=J(\beta)$.
This equivalence relation defines the \textit{ray class group of conductor $\fm$} which is denoted by
\[
\Cl(K,\fm)=I_\fm(K)/\sim.
\]

When $K=\QQ(\zeta_3)$, it is well-known that $\cO_{K}$ is a PID.
Moreover, note that $(\alpha)\sim(\beta)$ if and only if $\alpha\equiv\pm\zeta_3^n\beta\pmod\fm$ for some $n\in \ZZ_{\ge 0}$.

\begin{rem}
As will become clear later on, in the context of our problem we want to count the primes $N$ such that $(\fp)\sim(2)$ or $(4)$.   
\end{rem}

\subsubsection*{\underline{Recollections: Class Field Theory}}
Let $K'/K$ be a finite Galois extension of a number field $K$ with Galois group $\cH=\Gal(K'/K)$.
For a prime ideal $\fp$ in $K$ write $\fP \mid \fp$ to denote a prime ideal in $K'$.
Write $D_{\fP}$ to denote the decomposition group of $\fP$, which is defined as the following set
\[
D_{\fP} = \left\{ \sigma \in \cH \ : \ \sigma(\fP) = \fP\right\}.
\]
For $\sigma\in D_{\fP}$,
define $\overline{\sigma}\in\Gal((\cO_{K'}/\fP)/(\cO_{K}/\fp))$ such that
\[
\overline{\sigma}(x + \fP) = \sigma(x) + \fP.
\]
There exists a homomorphism
\[
\phi : D_\fP \longrightarrow\Gal((\cO_{K'}/\fP)/(\cO_{K}/\fp)) \textrm{ given by } \sigma \mapsto \overline{\sigma}.
\]
The kernel of this homomorphism is the \textit{inertia group} of $\fP$, which is denoted by $I_{\fP}$.

We now record well-known facts; see \cite[Chapter~3, Section~1]{Jan73} for proofs.

\begin{Th}
\label{FACTS}
With notation as set above, the following assertions hold.
\begin{enumerate}
  \item[\textup{(i)}] $\phi$ is surjective.
  \item[\textup{(ii)}] When $\fp$ is unramified in $K'$, the inertia group $I_\fP$ is trivial.
  \item[\textup{(iii)}] For $\alpha\in \cH$, the image $\alpha(\fP)$ is a prime ideal lying over $\fp$.
  \item[\textup{(iv)}] \label{fact 2.4} For $\alpha\in \cH$, the conjugate $\alpha D_\fP\alpha^{-1}=D_{\alpha(\fP)}$.
\end{enumerate}
\end{Th}

Suppose $\fp$ is unramified in $K'$.
Denote the element in $D_{\fP}$ that is mapped to the Frobenius map in $\Gal((\cO_{K'}/\fP)/(\cO_{K}/\fp))$ by $\sigma_{\fP/\fp}$. 
For the rest of the section assume that $\cH$ is abelian.

\begin{rem}
We can define $D_\fp=D_\fP$ which is independent of our choice of $\fP \mid \fp$ by Theorem~\ref{FACTS}(iv).
Likewise, we can also define $\sigma_\fp=\sigma_{\fP/\fp}$.
\end{rem}

Let $\fm$ be an ideal in $K$ such that all primes which ramify in $K'/K$ divide $\fm$.
Let $I \in I_\fm(K)$ with prime ideal decomposition $I=\prod_{i=1}^n\fp_i^{a_i}$.
Then there is a homomorphism $\Phi: I_\fm(K)\rightarrow \Gal(K'/K)$ that sends $I=\prod_{i=1}^n\fp_i^{a_i} \mapsto \prod_{i=1}^n \sigma_{\fp_i}^{a_i}$

\begin{Th}
Suppose $\fm$ is an ideal of $K$.
There exists a unique abelian extension $K' = K(\fm)$ such that $\Gal(K(\fm)/ K) \cong \Cl(K, \fm)$.
\end{Th}

\subsubsection*{\underline{Proving Theorem~\ref{Theorem B} using Chebotarev Density}}
Set $K=\QQ(\zeta_3)$ and $L=\QQ(\zeta_3,\sqrt[3]{N})$.
The goal is to count integral primes $N\not\equiv1\pmod9$ such that $\rk_3\left(\Cl(L)\right)=2$.

Recall that $N = \fn \overline{\fn}$ in $K$ and let $\fn=(\alpha)$ where $\alpha\in \cO_K$.\newline

\emph{Claim:} The condition in Theorem~\ref{criterion}(i) is equivalent to, 
\begin{equation}
\label{star}
\tag{$\star$} N\not\equiv 1\pmod9 \text{ and } \alpha\equiv\pm\zeta_3^v2^w \pmod{9}\text{ where } v\in\mathbb{Z} \text{ and } w\in\{1,2\}.
\end{equation}

\vspace{7pt}

\emph{Justification:} 
Up to units, 
\[
\alpha=\frac{A+3B\sqrt{-3}}{2}=\frac{A+3B}{2}+3B\zeta_3.
\]
If $3\mid B$, then $\alpha\equiv \frac{A}{2}\pmod{9}$.
Since $4N=A^2+27B^2$ and $N\not\equiv 1\pmod{9}$, one can check that 
\[
\frac{A}{2}\equiv \pm 2 \text{ or }\pm 4\pmod{9}.
\]
This implies \eqref{star}. \\
Conversely, assume \eqref{star}.
Then,  
\[
\frac{A}{2}\equiv \pm \zeta_3^v2^w\pmod{3} \textrm{ for some } v\in\mathbb{Z} \textrm{ and } w\in\{1,2\}.
\]
This forces $3 \mid v$, which implies that $\frac{A+3B}{2}\pm 2^w+3B\zeta_3\equiv 0\pmod{9}$.
Hence $3\mid B$ as desired.
\newline

We know from class field theory that
\[
\Cl(K,\fm)=(\cO_{K}/9\cO_{K})^\times/\langle -\zeta_3 \rangle.
\]
Therefore, it follows that $\abs{\Cl(K,\fm)}=9$.

Suppose that $\fm=(9)$.
By class field theory there exists a field $K'$ such that
\[
\Gal(K'/K)\cong \Cl(K,\fm)
\]
where the isomorphism (called the \emph{Artin map}) is given by sending $(\fn)$ to the Frobenius element $\sigma_{\fn}$.
The field $K'$ is called the ray class field.

The following fact is well-known and goes back to B.~Wyman; see for example \cite{wyman1973hilbert, cornell1988note}.
We provide a proof for the convenience of the reader.\newline

\emph{Claim:} $K'$ is Galois over $\QQ$ and $\Gal(K'/\QQ)=\Gal(K'/K) \rtimes \Gal(K/\QQ)$. \newline

\emph{Justification:}
Since $K'$ is closed under complex conjugates and the extension $K'/K$ is Galois, it follows that $K'$ is Galois over $\QQ$.
By the Schur--Zassenhaus theorem, the Galois group $\Gal(K'/\QQ)$ is the given semi-direct product. \newline 

Let $\phi$ be the natural map given by the fundamental theorem of Galois theory that sends
\[
\phi:\Gal(K'/\QQ) \longrightarrow \Gal(K/\QQ).
\]
Suppose that $\fn$ is a prime in $\cO_{K}$ lying over $N$ and $\fN\mid \fn$ is a prime ideal in $\cO_{K'}$.
Observe that $\sigma_{\fN/N}=(\sigma_{\fN/\fn}, \sigma_{\fn/N})$. 
If $\Phi$ is the Artin map described above, then $N$ has the property \eqref{star} if $\sigma_{\fN/N} = (\Phi((2)), e)$ or $(\Phi((4)), e)$, where $e$ is the identity in $\Gal(K/\QQ)$.
These two elements form a set fixed under conjugation since the non-trivial element of $\Gal(K/\QQ)$ is complex conjugation.

By the Chebotarev density theorem, the density of primes $N$ with $\sigma_{\fN/N}$ in a given conjugacy class.
Thus, the density of primes $N$ with property \eqref{star} is 
\[
\frac{\#\left\{(\Phi((2)), e),(\Phi((4)), e)\right\}}{\#\Gal(K'/\QQ)}=\frac{2}{18}=\frac{1}{9}.
\]

It follows that by restricting our counting to only primes of the form $1\pmod{3}$, we have
\begin{align*}
\PP\left(\rk_3\left(\Cl(L)\right)=2 \textrm{ \& } N\equiv 4,7 \pmod{9}\right) &= \frac{1}{3} \textrm{ and }\\ \PP\left(\rk_3\left(\Cl(L)\right)=1 \textrm{ \& } N\equiv 4,7\pmod{9}\right) &= \frac{2}{3}.
\end{align*}

\subsubsection{\textbf{The case when \texorpdfstring{$N\equiv 1\pmod{9}$}{}}}
\label{sec: 1 mod 9} 
In this section, we only provide partial results.
We begin by explaining some heuristics.

\subsubsection*{\underline{Heuristics}}
Set $N = 9z+1$ for some integer $z>0$.
Recall from Theorem~\ref{criterion}(ii) that $\rk_3 \left(\Cl(L)\right) = 2$ if and only if $( \frac{N-1}{3})!$ is a cubic residue modulo $N$.
In other words, $\rk_3 \left(\Cl(L)\right) = 2$ if and only if
\[
\left(\left( \frac{N-1}{3}\right)!\right)^{\frac{N-1}{3}} \equiv 1 \pmod{N}.
\]

To calculate $\rk_3\left(\Cl(L)\right)$ we compute the following the cubic Hilbert symbol modulo $N$, i.e.
\[
\left(\frac{(3z)!}{N}\right)_3 = \left(\frac{1}{N}\right)_3 \left(\frac{2}{N}\right)_3 \left(\frac{3}{N}\right)_3 \ldots \left(\frac{3z}{N}\right)_3.
\]
For a fixed integer $N \equiv 1 \pmod{9}$, exactly one-third of the residue classes are cubes modulo $N$, namely
\[
\left\{ 1, g^3 , g^6, \ldots, g^{N-4}\right\}
\]
where $\langle g \rangle = \left( \ZZ/N\ZZ\right)^\times \simeq \ZZ/(N-1)\ZZ$.
Therefore, expecting equi-distribution the probability that the cubic residue symbol $\left(\frac{(3z)!}{N}\right)_3 =1$ (resp. the cubic residue symbol is not $1$) is $\frac{1}{3}$ (resp. $\frac{2}{3}$).
As $N$ varies over all primes of the form $1\pmod{9}$ it is therefore reasonable to predict that 
\[
\PP(\rk_3\left(\Cl(L)\right)=2) = \frac{1}{3} \textrm{ and } \PP(\rk_3\left(\Cl(L)\right)=1) = \frac{2}{3}.
\]

Set $F = \QQ(\sqrt[3]{N})$ and write $\cCF = \Cl(F)[3^\infty]$ and $\cCL = \Cl(L)[3^\infty]$.
We first record an observation.

\begin{Lemma}
With notation introduced above,
\[
 \cCL \simeq \ZZ/3\ZZ \Longleftrightarrow \rk_3 \Cl(L)=1 \Longleftrightarrow \cCF \simeq \ZZ/3\ZZ.
\]
\end{Lemma}

\begin{proof}
Recall that if $\rk_3 \Cl(L)=1$ then $\cCL \simeq \ZZ/3\ZZ$; see \cite[p.~55]{Ger75-crelle}.
Also, we know that $\cCF \hookrightarrow \cCL$.
Combining the results of \cite{Calegari-Emerton, Ger05-BAusMS} we deduce that $\rk_3(\cCL) =1$ precisely when $3$ is the exact divisor of $\abs{\cCF}$.
The claim follows.
\end{proof}

As will become clearer in the remainder of this section, the key difficulty in proving the heuristics when $N\equiv 1 \pmod{9}$ lies in the fact that ambiguous classes do not always coincide with the strong ambiguous classes.
Moreover, the $p$-rank of class groups do not behave in a systematic way when this coincidence does not occur.

\subsubsection{\textbf{Ambiguous Class and Strong Ambiguous Class}}
\begin{definition}
\label{amb class and str amb class}
Let $\langle \sigma \rangle = \Gal(L/K) \simeq \ZZ/p\ZZ$.
Let $\mathcal{C}^{(\sigma)}$ denote the \textit{ambiguous ideal class} of $\mathcal{C}$ defined as
\[
\mathcal{C}^{(\sigma)} = \left\{ [\mathfrak{a}]\in \mathcal{C} \ : \ [\mathfrak{a}]^\sigma = [\mathfrak{a}]\right\}.
\]
The \textit{strong ambiguous ideal class} of $\mathcal{C}$, denoted by $\mathcal{C}_{\textrm{st}}^{(\sigma)}$ is defined as
\[
\mathcal{C}_{\textrm{st}}^{(\sigma)} = \left\{ [\mathfrak{a}]\in \mathcal{C} \ : \ \mathfrak{a}^{\sigma - 1} = (1)\right\}.
\]
\end{definition}
It is known (see for example, \cite[p.~161]{Ger87_crelle}) that
\[
\cCL^{(\sigma)} \simeq \cCL/\cCL^{1-\sigma}.
\]  
By \cite[Proposition~5.1]{Ger76}
\[
\rk_3(\cCL/\cCL^{1-\sigma}) = \#(\textrm{primes ramified in }L/K) - 1 = 2-1 = 1.
\]
In fact, we also know that $\cCL/\cCL^{1-\sigma}$ is an elementary abelian 3-group so
\[
\cCL^{(\sigma)} \simeq \cCL/\cCL^{1-\sigma} \simeq \ZZ/3\ZZ.
\]
The main difficulty in calculating $\rk_3\left(\Cl(L)\right)$ when $N\equiv 1 \pmod{9}$ is that the ambiguous classes \textbf{need not always} coincide with the strong ambiguous classes.
A criterion for this coincidence is precisely when $\zeta_3 \in \Norm_{L/K}(E_L)$; see for example \cite[Remark on p.~94]{Ger76}.
However, it has not been possible for us to give a quantitative answer to `how often' this coincidence occurs. 
Recall from Theorem~\ref{Norm} that $\zeta_3 \in \Norm_{L/K}(L^\times)$.
When the strong ambiguous class (denoted by $\mathcal{C}_{L,\textrm{st}}^{(\sigma)}$) \textit{does not coincide} with the ambiguous class
\[
\cCL^{(\sigma)} \simeq \mathcal{C}_{L,\textrm{st}}^{(\sigma)} \times \ZZ/3\ZZ. 
\]
In our case, $\rk_3 \cCL^{(\sigma)} =1$ which forces that $\mathcal{C}_{L,\textrm{st}}^{(\sigma)}$ must be trivial (in this exceptional case of non-coincidence) and (as we have observed before) that $\cCL^{(\sigma)} \simeq \ZZ/3\ZZ$.

\subsubsection*{\underline{When ambiguous classes are strong ambiguous}}

We analyze this situation of $N \equiv 1\pmod{9}$ using \cite[Theorem~5.3]{Ger76} which we have recorded before.

\begin{Lemma}
\label{coincidence impies rank 2}
If ambiguous ideal classes are strong-ambiguous then $\rk_3\left(\Cl(L)\right)=2$.
\end{Lemma}

\begin{proof}
When the hypothesis is satisfied, the matrix entries are determined by cubic Hilbert symbols, i.e.,
\[
\zeta_3^{\beta_{1j}} = \left( \frac{\fn_1 \overline{\fn_1}^2, N}{( \fn_j )}\right)_3 \textrm{ where }j=1,2.
\]
Our calculations in Lemma~\ref{Hilbert-pi} show that the cubic Hilbert symbol always takes the value 1.
In other words $\beta_{11} = \beta_{12} =0$.
Equivalently, the rank $s$ of this matrix is 0.
Gerth's theorem implies that whenever all the ambiguous classes coincide with strong-ambiguous classes, $\rk_3\left(\Cl(L)\right)=2$.
\end{proof}

\begin{rem}
If $\cCF \simeq \ZZ/3\ZZ$, then $\cCL^{(\sigma)} \neq \mathcal{C}_{L,\textrm{st}}^{(\sigma)}$.
\end{rem}

\subsubsection*{\underline{When ambiguous classes are not strong ambiguous}}
In this scenario, we see that $\rk_3\left(\Cl(L)\right)$ can take either the value 1 or 2.
We prove partial results which help analyze this situation.
Unfortunately, we are unable to obtain precise proportions which we had set out to prove.

For any extension of number fields $L/K$ we have the following exact sequence
\[
1 \longrightarrow \Cl(L/K) \longrightarrow \Cl(L) \xrightarrow{\Norm_{L/K}} \Cl(K).
\]
When $K = \QQ(\zeta_3) = \QQ(\sqrt{-3})$, the sequence is short exact since $\abs{ \Cl(K)} =1$; and $\Cl(L/K) \simeq \Cl(L)$.

By \cite[equation (1)]{Ger05-BAusMS}
\[
\rk_3\left(\Cl(L)\right) = \rk_3(\Cl(L/K)) = \rk_3(\cCL/\cCL^{1-\sigma}) + \rk_3(\cCL^{1-\sigma}/\cCL^{(1-\sigma)^2}).
\]
On the other hand, \cite[Proposition~5.1]{Ger76} asserts that the $3$-rank of the ambiguous ideal class of $L$ which is precisely $\rk_3(\cCL/\cCL^{1-\sigma})$ is given by
\[
\#(\textrm{primes ramified in }L/K) - 1 = 2-1 = 1.
\]

For ease of writing, henceforth write $\rk_3(\cCL^{1-\sigma}/\cCL^{(1-\sigma)^2}) = R$.
\cite[equation (2.14)]{Ger87_crelle} asserts
\begin{equation}
\label{R formula Gerth}
R = (\# \text{ramified primes in } L/K) - 1 - \rk M'_L - \epsilon = 1- \rk M'_L - \epsilon.
\end{equation}
Note $\pi = 1-\zeta_3$ is the unique prime above 3 in $K$ and $\fn_1, \fn_2 \equiv 1 \pmod{\pi^3 \cO_K}$, so $\epsilon = 0 \textrm{ 
or }1$.
More explicitly, $\epsilon = 0$ if the ambiguous ideal classes are strong-ambiguous.
However, if the ambiguous ideal classes are not strong-ambiguous then $\epsilon$ may take the value either $0$ or $1$.

We explain the construction of this matrix $M'_L$ following \cite{Ger87_crelle}.
The matrix $M'_L$ is a $1\times 3$ matrix with entries in $\mathbb{F}_3$ determined by the cubic Hilbert symbol.
Let $M'_L = [m'_{ij}]$ where $m'_{ij}\in \mathbb{F}_3$ with $1 \leq i \leq (\# \text{ramified primes in } L/K) - 1$, and $0 \leq j \leq (\# \text{ramified primes in } L/K)$.
The observation above implies that $i=1$ and $0\leq j \leq 2$.
The rank of the matrix $M'_L$ is 0 or 1.
Moreover,
\[
\zeta_3^{m'_{ij}} = \begin{cases}
\left( \frac{\zeta_3, N}{(\fn_1)}\right)_3 & \textrm{when }j=0 \\
\left( \frac{\fn_j, N}{(\fn_1)}\right)_3 & \textrm{when } 1 \leq j \leq 2. \end{cases}
\]
In view of calculations done in \cite[p.~92]{Ger76} note that 
\[
\left( \frac{\zeta_3, N}{(\fn_1)}\right)_3 = 1.
\]
Equivalently, $m'_{10}=0$.
Here is a way to check the calculations independently: by \cite[p.~165]{Ger87_crelle}
\[
\left( \frac{\zeta_3, N}{(\fn_1)}\right)_3 = \left( \frac{\zeta_3}{(\fn_1)}\right)_3^{-1},
\]
where the notation on the right side of the equality is the cubic residue symbol.
By definition
\[
\left( \frac{\zeta_3}{(\fn_1)}\right)_3 = \zeta_3^{n} \equiv \zeta_3^{\frac{\Norm(\fn_1) - 1}{3}} \pmod{(\fn_1)} \text{ for unique }n\in \{0,1,2\}.
\]
Note that in our case $-n = m'_{10}$.
Moreover, since the absolute norm $\Norm(\fn_1) = N \equiv 1 \pmod{9}$, 
\[
\left( \frac{\zeta_3}{(\fn_1)}\right)_3 \equiv \zeta_3^{\frac{N - 1}{3}} \equiv \zeta_3^{3k} \equiv 1 \pmod{(\fn_1)} \text{ for }n\in \{0,1,2\}.
\]
It follows that 
\[
M'_L = [0 \ m'_{11} \ m'_{12}].
\]

Finally, Lemma~\ref{Hilbert-pi} implies that $\left( \frac{\fn_1, N}{(\fn_1)}\right)_3 = 1$ so $m'_{11} = 0$.
This means 
\[
M'_L = [0 \ 0 \ m'_{12}].
\]

Recall that
\[
\left( \frac{\fn_2, N}{(\fn_1)}\right)_3 = \left( \frac{\fn_2, \fn_1}{(\fn_1)}\right)_3 \left( \frac{\fn_2, \fn_2}{(\fn_1)}\right)_3 = \left( \frac{\fn_1, \fn_2}{(\fn_1)}\right)_3^{-1} \left( \frac{\fn_2, \fn_2}{(\fn_1)}\right)_3 = 1.
\]
The penultimate equality follows from \cite[Chapter~V, Proposition~3.2(iv)]{Neu99}.
For the last equality, use the values of Hilbert symbols as calculated in Lemma~\ref{Hilbert-pi}.
This implies $m'_{12} = 0$, as well.
In particular, the matrix of interest $M'_L = [0 \ 0 \ 0]$.

\begin{rem}
We concluded that $\rk M'_L$ is always 0, irrespective of the value of $\epsilon$.
Rewrite \eqref{R formula Gerth} as
\[
R = 1-\epsilon.
\]
\end{rem}

When $N\equiv 1\pmod{9}$,
\[
\PP(\rk_3\left(\Cl(L)\right) =1) = \PP(R = 0 ) = \PP(\epsilon = 1) \leq \PP(\zeta_3 \not\in \Norm_{L/K}(E_L)).
\]
Equivalently,
\begin{align*}
\PP(\rk_3\left(\Cl(L)\right) =2) = \PP(R = 1) = \PP(\epsilon = 0) &= \PP(\zeta_3 \in \Norm_{L/K}(E_L)) + \PP(\epsilon = 0 \textrm{ \textbf{and} } \zeta_3 \not\in \Norm_{L/K}(E_L))\\
&= \PP(\cCL^{(\sigma)} = \mathcal{C}_{L,\textrm{st}}^{(\sigma)}) + \PP(\epsilon = 0 \textrm{ \textbf{and} } \cCL^{(\sigma)} \neq \mathcal{C}_{L,\textrm{st}}^{(\sigma)})
\end{align*}

\begin{rem}
Our computations predict that as $N$ varies over primes of the form $1\pmod{9}$, it is much more frequent to encounter the situation that the ambiguous classes are not strong ambiguous. 
\end{rem}


In what follows, assume that $3^n \Vert \abs{\cCF}$ where $n\geq 2$.

\begin{Lemma}
When the ambiguous ideal classes are \textit{not} strong-ambiguous and $3^n \Vert \abs{\cCF}$ with $n\geq 2$,
\[
\cCL \simeq \ZZ/3^{n-1}\ZZ \times \cCF \simeq \ZZ/3^{n-1}\ZZ \times \ZZ/3^n\ZZ.
\]
\end{Lemma}

\begin{proof}
Suppose that $\fP$ and $\overline{\fP}$ are the prime ideals of $L$ above $(\fn_1)$ and $(\overline{\fn_1})$.
As explained in \cite[p.~475]{Ger05-BAusMS}, when the ambiguous classes are \emph{not} strong-ambiguous the ideals $\fP$  and $\overline{\fP}$ are principal. 
This forces the existence of a positive \textit{odd} integer $j\geq 3$ such that $\cCL^{(\sigma)} \subseteq \cCL^{(1-\tau)^{j-1}}$ but not in $\cCL^{(1-\tau)^{j-1}}$.
It follows (see \cite[p.~474]{Ger05-BAusMS}) that
\[
\abs{\cCL} = \frac{1 \times \abs{\cCF}^2}{3}.
\]
This combined with \cite[Theorem~2]{Gerth75-MathComp} proves the claim.
\end{proof}

\section{\texorpdfstring{$p$}{}-Rank of the class group of \texorpdfstring{$L$}{} via Galois cohomology computations}
\label{sec: p rank via Gal cohom}

\emph{Basic Notation.}
Let $K=\QQ(\zeta_p)$ and $L = K( N^{1/p})$; set $\cG=\Gal(L/\QQ)$ and $G=\Gal(L/K)$.
Denote the quotient $\cG/G = \Delta = \Gal(K/\QQ)$.
Define the set $S=\{N, p, \infty\}$, write $\QQ_S$ to denote the maximal extension of $\QQ$ unramified outside $S$ and set $G_{\QQ, S} = \Gal(\QQ_S/\QQ)$ to denote the corresponding Galois group.
The absolute Galois group $\Gal(\overline{\QQ}/\QQ)$ is denoted by $G_{\QQ}$.

The main goal of this section is to provide sharper upper and lower bounds for the $p$-rank of $L$.
The first main result towards attaining the goal is proving Theorem~\ref{main theorem Rusiru note 3} where we characterize unramified extensions of $L$ with specified Galois group structure and then using combinatorics, provide an explicit count of such extensions in Corollary~\ref{define rij}.
What the corollary says is that the number of extensions can be written explicitly in terms of the dimension of certain Galois cohomology groups. 
Next, we prove an abstract formula for the $p$-rank of $\Cl(L)$ in Theorem~\ref{main theorem of note 4} involving the sum of dimensions of the Galois cohomology groups.
These individual terms are difficult to compute explicitly.
However, building on \cite{SS19}, we are able replace the above cohomology groups by more manageable ones, i.e. cohomology groups with coefficients in (twists of) $\Fp$ whose dimensions are relatively easy to compute.
This comes at the cost that we are no longer able to prove an exact formula, but instead have inequalities; see Theorems~\ref{main result - revised lower bounds} and \ref{th: better upper bounds}.

\subsection{Unramified Extensions of \texorpdfstring{$L$}{}}
In this section we provide an explicit count of how many unramified Galois extensions with specified Galois groups the number field $L$ can have.

Consider the mod $p$ cyclotomic character
\[
\chi: G_{\QQ, S} \longrightarrow \mathbb F_p^{\times} 
\]
and let $b$ denote the map
\[
b: G_{\QQ, S} \longrightarrow \langle \zeta_p \rangle  \textrm{ 
given by } \sigma \mapsto \frac{ \sigma(N^{1/p})}{N^{1/p} }.
\]
We often consider $\langle \zeta_p \rangle \simeq \Fp(1)$ and view the image of $b$ in this additive group.
Note that $b$ and $\chi$ are trivial on the absolute Galois group $G_{L} = \Gal(\overline{\QQ}/L)$ and so can be thought of as defined on $\cG$.
Observe that $\chi|_G$ is trivial.

Set $\Fp(i)$ to denote the module $\Fp$ on which $G_{\QQ,S}$ acts via $\chi^i$.
Next define $V\simeq \Fp^2$ to be the vector space on which $G_{\QQ,S}$ acts via the representation
\begin{equation}
\begin{split}
\label{rep}
G_{\QQ,S} &\longrightarrow \GL_2(\Fp)\\
\sigma & \mapsto \begin{pmatrix}
    \chi(\sigma) & b(\sigma) \\
    0 & 1\\
  \end{pmatrix}.
\end{split}
\end{equation}
For ease of notation, henceforth denote the $G_\QQ$-representation $\Sym^{i}(V)\otimes \Fp(j)$ by $A^{i,j}$.
Pick a basis for $A^{i,j}$ such that the $G_{\QQ,S}$-representation
\[
\rho_{ij}: G_{\QQ,S} \longrightarrow \GL_{i+1}(\Fp)
\]
is given by the matrix
\[
[\rho_{ij}] := \begin{pmatrix}
\chi^{i+j} & \chi^{i+j-1}b & \chi^{i+j-2}\frac{b^2}{2} & \ldots & \chi^{j}\frac{b^i}{i!}\\
 & \chi^{i+j-1} & \chi^{i+j-2}b & \ldots & \vdots \\
 & & \chi^{i+j-2} & \ldots & \vdots \\
 & & \ddots & \vdots & \vdots\\
 & & & & \chi^j
\end{pmatrix}
\]
Observe that $\rho_{ij}$ is trivial on $G_L$ and so there is an induced action of $\mathcal G$ on $A^{i, j}$.  
With this setup, there is an injection of $G_\QQ$-representations
\[
A^{i-1, j+1} \hookrightarrow A^{i, j}.
\]
Thus, we obtain a filtration
\begin{equation}
\label{filtration}
A^{0, j+i} \subseteq A^{1, j+i-1} \subseteq \cdots \subseteq A^{i-1, j+1} \subseteq A^{i, j} \subseteq A^{i+1, j-1} \cdots 
\end{equation}


In the following proposition we describe the centralizer of the set of matrices $[\rho_{ij}](g)$ for all $g\in G_{\QQ}$, which is henceforth denoted by $C([\rho_{ij}], G_{\QQ})$.
As is be clarified in the proof below, the centralizer has a different description based on whether $i<p-1$ or $i=p-1$.

\begin{prop}
\label{Rusiru Note 4 Prop 2}
With notation as introduced above and writing $\Id$ to denote the identity matrix,
\[
C([\rho_{ij}], G_{\QQ}) = \begin{cases}
  \Fp^\times \Id_{i+1} \textrm{ when } i< p-1\\
  \left\{ \begin{pmatrix}
    \lambda & 0 & \ldots & c\\
    0 & \lambda & \ddots & 0\\
    \vdots & \vdots &\ddots & \vdots\\
    0 & 0 & \ldots & \lambda
  \end{pmatrix} \textrm{ such that }\lambda\in \Fp^\times, c\in \Fp\right\} \textrm{ when } i = p-1.
\end{cases}
\]
\end{prop}

\noindent \textbf{Notation:} The matrix in the case $i=p-1$ will be denoted by $\mathcal{M}(\lambda, c)$ for ease of notation.

\begin{proof}
\uline{Case 1: When $i < p-1$.} \newline

Let $h$ be an element in $G_{F}$ where $F= \QQ(N^{1/p})$.
Then $b(h) = 0$. 
This implies that $\rho(h)$ is a diagonal matrix with distinct entries.
The only matrices that commute with such matrices are diagonal matrices.
Therefore,
\[
C([\rho_{ij}], G_F) \subseteq \{\textrm{Diag}(\lambda_1 , \; \ldots , \; \lambda_{i+1}) \ \mid \ \lambda_t\in \Fp^{\times} \textrm{ for all }t\}.
\]
On the other hand if $g\in G_{\QQ}\setminus G_F$, then $\rho_{ij}(g)$ is an upper triangular matrix with non-zero entries above the diagonal.
The only non-zero diagonal matrices that commute with such a $\rho_{ij}(g)$ are (non-zero) scalar multiples of $\Id_{i+1}$.
This completes the proof when $i<p-1$. \newline

\uline{Case 2: When $i = p-1$.} \newline

When $i=p-1$ then notice that for $h\in G_F$, the matrix $\rho_{ij}(h)$ is a diagonal matrix but not all entries are distinct.
Indeed, the (1,1)-entry is $\chi^{p-1+j}(h) = \chi^j(h)$ and the ($i+1$, $i+1$)-entry is $\chi^{j}(h)$.
In this case the set of matrices that commute with $\rho_{ij}(g)$ are given by
\[
\left\{ \begin{pmatrix}
\lambda_1 & 0 & \ldots & x\\
    0 & \lambda_2 & \ddots & 0\\
    \vdots & \vdots &\ddots & \vdots\\
    y & 0 & \ldots & \lambda_p  
\end{pmatrix} \textrm{ where } \lambda_t\in \Fp^\times, \ x,y\in \Fp
\right\} =: M((\lambda_t), x, y).
\]
If $g\in G_{\QQ}\setminus G_F$, and if we further require that
\begin{equation}
\label{centralizer eqn}
 \rho(g) \cdot M((\lambda_t), x, y)=M((\lambda_t), x, y) \cdot \rho(g)
\end{equation}
then upon comparing the (1,1)-entry, we note that
\[
\chi^{j}(g)\lambda_1 + \frac{b^{p-1}(g)}{(p-1)!}\chi^j(g)y = \chi^{j}(g)\lambda_1.
\]
This implies $y=0$.
Comparing the upper and lower $(p-1)\times (p-1)$-blocks of both sides of \eqref{centralizer eqn},
\[
\lambda_1 = \lambda_2 = \ldots = \lambda_p = \lambda \textrm{ (say)}.
\]
The claim of the proposition now follows immediately.
\end{proof}

\begin{rem}
The centralizer elements parameterize the different bases one can choose to obtain the same matrices $\rho_{ij}(g)$.
In other words, if $\{\beta_0, \ldots, \beta_i \}$ is a basis such that $A^{i,j}$ is given by $\rho_{ij}(g)$ then $\{\gamma\beta_0, \ldots, \gamma\beta_i \}$ is also such a basis where $\gamma$ is a centralizer element.
\end{rem}

\begin{prop}
In view of \eqref{filtration}, $A^{i, j}$ contains exactly one copy of $A^{i-\alpha, j+\alpha}$ for all $0 \leq \alpha \leq i$.
\end{prop}

\begin{proof}
This statement will be proven by induction on $i$.
We begin with the observation that in view of the above filtration $(A^{i,j})^G = A^{0,j+i}$.
\newline

\underline{Base case: $i=1$}
We need to verify that $A^{1,j}$ contains exactly one copy of $A^{0,j+1}$.
Suppose that $X \subseteq A^{1,j}$ and $X\cong A^{0,j}$.
Note that 
\[
X^G \subseteq (A^{1,j})^G = A^{0, j+1}.
\]
But, $X^G = X$ and $A^{0, j+1}$ are both 1-dimensional.
Thus, $X^G = A^{0, j+1}$ as desired.
\newline

\underline{Induction hypothesis:} Assume that the result holds for all $0 \leq r <i$.
\newline

\underline{Induction step:}
Suppose that $X \subseteq A^{i,j}$ and $X \cong A^{i-\alpha, j+\alpha}$.
Then using the fact that $X^G = (A^{i-\alpha,j+\alpha})^G = A^{0, i+j}$ we obtain that
\[
\frac{X}{X^G} \subseteq \frac{A^{i,j}}{(A^{i,j})^G} \cong A^{i-1, j},
\]
where the isomorphism is obtained by forgetting the first coordinate.
Similarly,
\[
\frac{X}{X^G} \cong A^{i-1 - \alpha, j+\alpha}.
\]
By the induction hypothesis, there is only one possibility for $\frac{X}{X^G}$ and hence also for $X$.
\end{proof}

Next we include the Selmer condition $\Lambda$ which will play a crucial role throughout our paper.
This is the same definition as in \cite[Section~3.2]{SS19}

\begin{definition}
\label{Lambda Selmer condn}
Let $A$ be a $G_{\QQ}$-module.
Define $\Lambda = \{L_v\}$ to be the \textit{Selmer condition} given by 
\begin{itemize}
  \item$L_{\ell} =H^1_{\ur}(G_{\QQ_{\ell}} ,A) = H^1(G_{\QQ_{\ell}}/I_{\ell}, A^{I_\ell})$ for $\ell \neq N,p$ where $I_\ell$ is the inertia group.
  \item $L_N =H^1(G_{\QQ_N} ,A)$
  \item $L_p = {\Res}^{-1}\left(H^1_{\ur}(G_{L_p} ,A)\right)$ where $\Res$ is the restriction map $H^1(G_{\QQ_p} ,A) \rightarrow H^1(G_{L_p} ,A)$.
\end{itemize}
The \textit{Selmer group} associated to the Selmer condition $\Lambda$ is defined as
\[
H^1_{\Lambda}(G_{\QQ}, A) = \ker\left(H^1(G_{\QQ},A) \longrightarrow \prod_{v} \frac{H^1(G_{\QQ_v}, A)}{L_v}\right).
\]
\end{definition}

\begin{rem}
Throughout this section, $p$ is assumed to be odd.
Hence $H^1(G_{\mathbb{R}},A^{i, j}) = 0$ for all $i$ and $j$. 
Thus, we need not specify a local condition at the infinite place.
\end{rem}

\begin{prop}
\label{res is inj}
Consider the restriction map 
\[
\Res: H^1_{\Lambda}(G_{\QQ}, A^{i, j}) \longrightarrow H^1(G_L, A^{i, j}).
\]
\begin{enumerate}
\item[\textup{(i)}] If $j \not\equiv 1 \pmod{p-1}$ or $i=p-1$, then $ \Res$ is injective.
\item[\textup{(ii)}] If $j \equiv 1 \pmod{p-1}$ and $i< p-1$, then $\ker( \Res)$ is one-dimensional.
\end{enumerate}
\end{prop}

\begin{proof}
By the inflation-restriction sequence
\[
0 \longrightarrow H^1(\cG, A^{i, j}) \longrightarrow H^1(G_\QQ, A^{i, j}) \longrightarrow H^1(G_L, A^{i, j})^{\cG}.
\]
We also have
\[
0 \longrightarrow H^1(\Delta, (A^{i, j})^G) \longrightarrow H^1(\cG, A^{i, j} ) \longrightarrow H^1(G, A^{i, j})^{\Delta} \longrightarrow H^2(\Delta, (A^{i, j})^G).
\]
Since $\Delta = \Gal(\QQ(\zeta_p)/\QQ)$ has order prime to $p$, the first and the last terms vanish.
Hence, 
\begin{equation}
\label{Rusiru note 3 equation 1}
H^1(\cG, A^{i, j}) \cong H^1(G, A^{i, j})^{\Delta}.
\end{equation}
For the first statement of the proposition, we use the next lemma (see Lemma~\ref{next lemma}) to conclude,
\[
H^1(G, A^{i, j})^{\Delta}=0.
\]
Thus, $H^1(\cG, A^{i, j})$ vanishes as desired.

Now if $j \equiv 1 \pmod{p-1}$ and $i<p-1$, note that 
\[
H^1(G, A^{i,j})^{\Delta} \simeq \Fp
\]
and the cohomology group is generated by a 1-cocycle $c : g \mapsto \uline{\boldsymbol{b}}^{(i)}{(g)}$; see below for notation.
But via the isomorphism induced by restriction in \eqref{Rusiru note 3 equation 1}, a 1-cocycle $c' : g' \mapsto \uline{\boldsymbol{b}}^{(i)}{(g')}$ in $H^1(\cG, A^{i,j})$ maps to a 1-cocycle $c$.
\end{proof}

\begin{Lemma}
\label{next lemma}
Suppose that $i<p-1$.
There is an isomorphism $H^1(G,A^{i, j}) \cong \mathbb F_p(1-j)$ and the cohomology group is generated by the 1-cocycle
\[
c : g \mapsto \begin{pmatrix}
      \frac{b^{i+1}}{ (i+1)! }(g)  \\
      \vdots \\
      b(g) 
     \end{pmatrix} .
\]
When $i=p-1$, the cohomology group is trivial.
\end{Lemma}

\noindent \textbf{Notation:} Write $\uline{\boldsymbol{b}}^{(i)}$ to denote the vector $\begin{pmatrix}
  \frac{b^i}{i!}\\
  \frac{b^{i-1}}{(i-1)!} \\
  \vdots\\
  b
\end{pmatrix}$ and $\uline{\boldsymbol{b}}^{(i)}(g)$ to denote $\begin{pmatrix}
  \frac{b^i}{i!}(g)\\
  \frac{b^{i-1}}{(i-1)!}(g) \\
  \vdots\\
  b(g)
\end{pmatrix}$.

\begin{proof}
We first show that the 1-cocycle $c : g \mapsto \uline{\boldsymbol{b}}^{(i)}{(g)}$ generates $H^1(G, A^{i,j})$.
As $G$ is cyclic, suppose that it is generated by some element $g_0$.
We know that $A^{i, j}$ is finite; now by \cite[Chapter 2, Example 1.20]{milneCFT}, there exists an isomorphism 
\begin{align}\label{Milne Isomorphism}
H^1(G, A^{i, j}) &\cong \frac{\ker(\Norm(G) )}{ (\Id -{g_0}) A^{i, j} }\\
[\sigma] &\mapsto [\sigma(g_0)]
\end{align}
where $\Norm(G)$ is the norm map defined by
\[
\Norm(G) := \Id + g_0 + \cdots +g_0^{p-1}
\] 
and $g_0$ is being viewed as a matrix of size $(i+1) \times (i+1)$.
When $i=p-1$, note that the right side of the above isomorphism is 0-dimensional because the kernel of the norm is $i$-dimensional (in this case) and $(\Id -{g_0}) A^{i, j} $ is always $i$-dimensional.
This proves the last assertion of the lemma.

In the following, we focus on the case that $i<p-1$.
A direct computation shows that $\Norm(G) = \Id + g_0+ \cdots +g_0^{p-1}= 0$ when $i<p-1$.
Hence, $\ker(\Norm(G))=\Fp^{i+1}$ and $\rank(\Id-g_0)=i$.
Furthermore, $\uline{\boldsymbol{b}}^{(i)}{(g_0)} \not\in (\Id - g_0)A^{i,j}$.
Hence, $H^1(G, A^{i, j})$ is one-dimensional.

We now prove the first assertion of the lemma, namely the isomorphism.
Observe that the action of $\cG/G$ on $H^1(G, A^{i,j})$ can be understood by how $x \in \cG/G$ acts on $b(g_0)$, which is the last entry of the representative $\uline{\boldsymbol{b}}^{(i)}{(g_0)}$ of $\frac{\ker(\Norm(G) )}{ (\Id -{g_0}) A^{i, j} }$.
Note that
\[
x \cdot b(g_0) = \chi^j(x^{-1})b(\widetilde{x} g_0 \widetilde{x}^{-1}).
\]
Here $\widetilde{x}\in \cG$ is a lift of $x$.
Since $G$ is abelian, the conjugate $\widetilde{x} g_0 \widetilde{x}^{-1}$ is independent of the choice of the lift.
The above action is understood via the isomorphism in \eqref{Milne Isomorphism}.
Set $\epsilon:=b(g_0) = \frac{g_0(N^{1/p})}{N^{1/p}} $, i.e., equivalently write
\[
g_0(N^{1/p}) = \epsilon N^{1/p}.
\]
Also, suppose that $\widetilde{x}^{-1}(N^{1/p}) = \zeta N^{1/p}$; here, $\zeta$ is some $p$-th root of unity and we remind the reader that $\widetilde{x}$ is as an automorphism.
Then $N^{1/p} = \widetilde{x}(\zeta N^{1/p})$.
Moreover,
\begin{align*}
b(\widetilde{x} g_0 \widetilde{x}^{-1}) &:= \frac{(\widetilde{x} g_0 \widetilde{x}^{-1})(N^{1/p})}{N^{1/p}} \\
&= \frac{(\widetilde{x} g_0)(\zeta N^{1/p}) }{N^{1/p}} = \frac{\widetilde{x}(\zeta g_0(N^{1/p}))}{N^{1/p}} = \frac{x(\zeta) \times (\widetilde{x} g_0)(N^{1/p})}{N^{1/p}} = x(\zeta) \times \frac{\widetilde{x}(\epsilon N^{1/p})}{N^{1/p}}\\
& = \frac{x(\epsilon) \times \widetilde{x}(\zeta N^{1/p})}{N^{1/p}} = x(\epsilon) = \chi(x) \epsilon\\
& = \chi(x) b(g_0).
\end{align*}
Thus, $x \cdot b(g_0) = \chi^{1-j}(x) b(g_0)$.
This gives the desired isomorphism.
\end{proof}

Next, we state and prove the main result of this section; but first we present a lemma and introduce two definitions. 

\begin{Lemma}
\label{Semidirect-product}
Let $M$ be an elementary abelian unramified extension of $L$ which is Galois over $\QQ$.
Then the following short exact sequence splits, 
\[
1\longrightarrow \Gal(M/L)\longrightarrow \Gal(M/\QQ)\longrightarrow \cG\longrightarrow 1.
\]
Hence, $\Gal(M/\QQ)\cong \Gal(M/L)\rtimes \cG$.

\begin{proof}
The proof is similar to that of \cite[Lemma~3.1.3]{SS19}.
\end{proof}
    
\end{Lemma}

\begin{definition}\leavevmode{}
\begin{enumerate}
    \item[(a)] An \emph{ $A^{i,j}$ extension} of $L$ is a Galois extension $M/L$ satisfying  $\Gal(M/L) =: \uline{A}$ such that $\underline{A} \simeq A^{i,j}$ as a $\cG$-module.
    \item[(b)] Given a vector space $V$ over a field $F$, the projective space $\PP^1(V)$ is the set of equivalence classes of $V \setminus \{0\}$ under the equivalence relation $\sim$ defined by $x \sim y$ if there is a nonzero element $\lambda$ of $F$ such that $x = \lambda y$.
\end{enumerate}
\end{definition}

\begin{Th}
\label{main theorem Rusiru note 3}
When $1 \leq i< p-1$ and $j\neq 1$, there is the following one-to-one correspondence 
\[
\{\textup{unramified $A^{i, j}$ extensions of } L \} \xleftrightarrow{1:1} \mathbb P^1 H^1_\Lambda(G_{\QQ},A^{i, j})\setminus \mathbb P^1 H^1_\Lambda(G_{\QQ}, A^{i-1, j+1}).
\]
\end{Th}

\begin{proof}
Let $M/L$ be an unramified $A^{i,j}$ extension where $\uline{A}:=\Gal(M/L)$.
Pick a basis in such a way that the action is given by the matrix $[\rho_{ij}]$; this basis is unique up to multiplication by $\mathbb F_p^{\times}$.
We have the Galois restriction map
\[
\rho: G_{\QQ, S} \longrightarrow \Gal(M/\QQ) \cong \begin{pmatrix}
\cG & \underline{A} \\
0 & 1 \\
\end{pmatrix}, 
\]
where Lemma \ref{Semidirect-product} is used to interpret $\Gal(M/\mathbb Q)$ as the given block-matrix group.
With the fixed basis on $\uline{A}$,
\[
\cG \hookrightarrow \Aut(\underline{A}) \cong \GL_{i+1}(\mathbb F_p).
\]
So, we get a representation 
\[
\rho: G_{\QQ, S} \longrightarrow \begin{pmatrix}
\cG & \underline{A} \\
0 & 1 \\
\end{pmatrix} \subseteq \GL_{i+2}(\mathbb F_p).
\]
Consider the 1-cocycle
\[
a_{\underline A}: G_{\QQ, S} \longrightarrow A^{i, j}
\text{ given by } g \mapsto  \underline{\boldsymbol{a}}^{(i)}(g)=\begin{pmatrix}
      a_1(g)  \\
      \vdots \\
      a_{i+1}(g) 
     \end{pmatrix}  
\]
where $\underline{\boldsymbol{a}}^{(i)}(g)$ is the upper right column vector of the matrix $\rho(g)$.
This determines a class $[a_{\underline A} ] \in H^1(G_\QQ, A^{i, j})$.
Since the extension $M/L$ is unramified, one can check the Selmer conditions to conclude that $[a_{\underline A}] \in H^1_{\Lambda}(G_{\QQ}, A^{i, j})$.\newline

\emph{Claim:} $[a_{\underline A}]\in H^1_\Lambda(G_{\QQ}, A^{i, j})\setminus H^1_\Lambda(G_{\QQ}, A^{i-1, j+1})$.\newline

\emph{Justification:}
We have the following diagram
\[
\begin{tikzcd}
H^1_\Lambda(G_{\QQ}, A^{i-1, j+1}) \arrow{r}{\psi} \arrow[swap]{d}{\Res} & H^1_\Lambda(G_{\QQ}, A^{i, j}) \arrow{d}{\Res} \\%
H^1(G_{L}, A^{i-1, j+1})^{\cG} \arrow{r}{ }& H^1(G_{L}, A^{i, j})^{\cG}
\end{tikzcd}.
\]
Since the Galois group $G_L$ acts trivially on $A^{i, j}$,
\[
H^1(G_{L, S}, A^{i, j}) \cong \Hom(G_{L, S}, A^{i, j}).
\]
Therefore, the (surjective) Galois restriction map
\[
G_{L, S} \longrightarrow \Gal(M/L) = \underline{A} \cong A^{i, j}
\]
induces the (surjective) map 
\[
\Res [a_{\underline A}]: G_{L,S} \longrightarrow A^{i, j}.
\]
Note that $ \Res[a_{\underline A}] \not \in \Image(\psi)$; indeed otherwise $\Image( \Res[a_{\underline A}]) \subseteq A^{i-1, j+1} \subseteq A^{i, j}$ which contradicts surjectivity.
Hence, 
\[
[a_{\underline{A}}] \in H^1(G_{\QQ}, A^{i, j}) \setminus H^1_\Lambda(G_{\QQ}, A^{i-1, j+1}).
\]
We have defined a multi-valued map 
\begin{equation}
\label{multi-valued map}
M_{\uline{A}} \mapsto [a_{\uline{A}}].
\end{equation}
Taking the projective space will make the ambiguity of basis of $\uline{A}$ irrelevant thereby defining an element in $\mathbb P^1 H^1_\Lambda(G_{\QQ},A^{i, j})\setminus \mathbb P^1 H^1_\Lambda(G_{\QQ}, A^{i-1, j+1})$.
\newline

We now proceed to define the inverse map.
Pick $y \in \mathbb P^1 H^1_\Lambda(G_{\QQ}, A^{i, j}) \setminus \mathbb P^1 H^1_\Lambda(G_{\QQ}, A^{i-1, j+1})$ and a representative $[x] \in H^1_\Lambda(G_{\QQ}, A^{i, j}) \setminus H^1_\Lambda(G_{\QQ}, A^{i-1, j+1})$.
Then we have a $\cG$-equivariant morphism
\[
\Res[x]: G_{L, S} \longrightarrow A^{i, j}.
\]
Next, define
\[
M_y :=\QQ_S^{\ker(\Res[x])}.
\]
Note that this definition of $M_y$ does not depend on the choice of the representative, so it is well-defined.
In view of the natural inclusion map,
\[
\Gal(M_y/L) \cong \Image(\Res[x]) \cong A^{i-\alpha, j+\alpha} \subseteq A^{i, j} \text{ for some }\alpha \geq 0.
\]

\emph{Claim:} The extension $M_y/ L$ is unramified.
\newline

\emph{Justification:}
We check this locally at each prime $\ell$.
Set $I_{S, \ell}$ and $I_{L, S, \ell}$ to denote the ramification group in the extension $\QQ_S/\QQ$ and $\QQ_S/L$, respectively.
We have the following commutative diagram
\[ 
\begin{tikzcd}
H^1(G_{\QQ,S}, A^{i,j}) \arrow{r}{\Res} \arrow[swap]{d}{\psi_{\ell}} & H^1(G_{L,S}, A^{i,j}) \arrow{d}{\psi_{L,\ell}} \\%
H^1(I_{S,\ell}, A^{i,j}) \arrow{r}{\Res_{\ell}}& H^1(I_{L,S,\ell}, A^{i,j}).
\end{tikzcd}
\]

Note that $[x] \mapsto \Res[x]$ under the top horizontal (restriction) map.
When $\ell\neq p, N$, we have that $I_{S, \ell}$ is trivial and 
\[
\psi_{L,\ell}(\Res[x])=\Res_\ell(\psi_\ell[x])=\Res_\ell(0)=0.
\]
On the other hand when $\ell=p$, it follows by definition of the local conditions at $p$ in $\Lambda$ that $\psi_{L,\ell}(\Res[x])=0$.
Finally, in the case that $\ell=N$, it follows from \cite[Proposition~2.2.2]{SS19} that $\Image(\Res_{\ell})=0$.
Hence, $M_y/L$ is unramified as desired.
\newline 

We now show that $M_y/L$ is an $A^{i, j}$ extension i.e., $\text{Res}[x]$ is surjective.\newline

\emph{Claim:} $\text{Res}[x]$ is surjective.
\newline

\emph{Justification:}
Suppose that $\Res[x]$ is not surjective.
As noted above $\Image(\text{Res}[x]) \cong A^{i-\alpha, j+\alpha}$ where $\alpha >0$.
Define the map \eqref{multi-valued map} such that
\[
M_y \mapsto [a_{M_y}] \in H^1_\Lambda(G_{\QQ}, A^{i-\alpha, j+\alpha})
\]
Consider the following diagram
\[ 
\begin{tikzcd}
H^1_{\Lambda}(G_{\QQ}, A^{i-\alpha,j+\alpha}) \arrow{r}{\psi_1} \arrow[swap]{d}{\Res'} & H^1_{\Lambda}(G_{\QQ}, A^{i-1,j+1}) \arrow{r}{\psi_2} & H^1_{\Lambda}(G_{\QQ}, A^{i,j}) \arrow{ld}{\Res}\\%
H^1(G_{L}, A^{i-\alpha,j+\alpha}) \arrow{r}{\psi_3} & H^1(G_{L}, A^{i,j}) & {} 
\end{tikzcd}
\]
An easy diagram chase shows that
\[
\Res(\psi_2 (\psi_1([a_{M_y}]))) = \psi_3(\Res'([a_{M_y}])) = \Res[x].
\]
Since $j \not \equiv 1 \pmod{p-1}$, the restriction map $\Res$ is injective; see Proposition~\ref{res is inj}.
It follows that $[x]\in \Image(\psi_2)$ but this is a contradiction.
\newline

It is now easy to check that the natural inverse of \eqref{multi-valued map} is 
\[
y \mapsto M_y.
\qedhere
\]
\end{proof}

\begin{rem}\leavevmode
\begin{enumerate}
\item[(a)] When $i=0$, note that $A^{0,j} = \Fp(j)$.
Since the Galois group $G$ acts trivially on $\uline{A}$, there is a natural map
\[
\uline{A} \rtimes \cG \longrightarrow \uline{A} \rtimes (\cG/G)
\]
where $\cG/G$ acts as $\chi^j$ on $\uline{A}$.
By fixing an isomorphism $\uline{A} \cong \Fp(j)$, we have a representation
\[
\overline{\rho}: G_{\QQ, S} \longrightarrow \uline{A} \rtimes \cG \longrightarrow \uline{A} \rtimes (\cG/G) \cong \begin{pmatrix}
\cG/G & \underline{A} \\
0 & 1 \\
\end{pmatrix} \subseteq \GL_2(\Fp)
\]
given by the map
\[
g \mapsto \begin{pmatrix}
\chi^j & a(g) \\
0 & 1 \\
\end{pmatrix}.
\]
Consider the 1-cocycle $a_{\uline{A}} : G_{\QQ, S} \rightarrow A^{0,j}$ given by $g \mapsto a(g)$.
This produces a class $[a_{\uline{A}}]\in H^1(G_{\QQ}, A^{0,j})$ and the rest of the proof goes through verbatim.
\item[(b)] When $j\equiv 1\pmod{p-1}$ and $i<p-1$ we replace $H^1_\Lambda(G_{\QQ},A^{i, 1})$ by $H^1_\Lambda(G_{\QQ},A^{i, 1})/\langle \uline{\boldsymbol{b}}^{(i)} \rangle$.
This induces the injection map
\[
\overline{\Res}: H^1_\Lambda(G_{\QQ},A^{i, 1})/\langle \uline{\boldsymbol{b}}^{(i)} \rangle \longrightarrow H^1(G_L, A^{i,1})
\]
and the rest of the proof goes through verbatim.
\end{enumerate}
\end{rem}

In the proof of Theorem~\ref{main theorem Rusiru note 3} we constructed a multi-valued map $M_{\uline{A}} \mapsto [a_{\uline{A}}]$ where $[a_{\uline{A}}]$ is an element in $H^1_\Lambda(G_{\QQ}, A^{i, j})\setminus H^1_\Lambda(G_{\QQ}, A^{i-1, j+1})$.
When $i < p-1$, the change of basis is possible precisely by scalar multiplication.
This is why the map can be converted into a function by taking the projective spaces.

Now suppose that $i=p-1$.
For a random choice of basis $\uline{\beta}$ suppose that the class
\[
[a_{\uline{A}}] = [g \mapsto \underline{\boldsymbol{a}}^{(i)}{(g)}]
\]
where as before $[a_{\uline{A}}]\in H^1_\Lambda(G_{\QQ}, A^{i, j})\setminus H^1_\Lambda(G_{\QQ}, A^{i, j+1})$.
A different basis $\uline{\beta'}$ gives the class 
\[
[a'_{\uline{A}}] = [g \mapsto \mathcal{M} \underline{\boldsymbol{a}}^{(i)}{(g)}]
\]
where $\mathcal{M} = \mathcal{M}(\lambda, c)\in C([\rho_{ij}], G_{\QQ})$.
Note that $a_{i+1}$ is not the zero map.
Consider the restriction map
\[
\Res: H^1_\Lambda(G_{\QQ}, A^{i, j}) \longrightarrow H^1_\Lambda(G_{L}, A^{i, j})^{\cG}
\]
Then it follows from a simple calculation that
\[
\Res(a_{\uline{A}} - a'_{\uline{A}}) = \Res\left[g \mapsto \begin{pmatrix}
  \left( (1-\lambda)a_1 + ca_{i+1} \right)(g)\\
  ((1-\lambda)a_2)(g)\\
  \vdots \\
  ((1-\lambda)a_{i+1})(g)
\end{pmatrix} \right] \neq 0
\]
unless $\mathcal{M} = \Id$.
What this means is that distinct choices of $\mathcal{M}$ give distinct classes and the map $\uline{A} \mapsto [a_{\uline{A}}]$ is a $1:p(p-1)$-map which is surjective.

We can now record a corollary of the above theorem.

\begin{cor} \label{cor1}
\label{define rij}
For $i\geq 0$, define 
\[
r_{ij} = \begin{cases}
\dim_{\Fp}\left( H^1_\Lambda(G_{\QQ}, A^{i,j})\right) & \textrm{ when } j\not\equiv 1 \pmod{p-1} \textrm{ or } i = p-1 \\
\dim_{\Fp}\left( H^1_\Lambda(G_{\QQ}, A^{i,j})\right) - 1 & \textrm{ when } j\equiv 1 \pmod{p-1} \textrm{ and } i < p-1.
\end{cases}.
\]
If $i<0$, set $r_{ij}=0$.
The number of unramified $A^{i,j}$ extensions of $L$ is given by
\[
\begin{cases}
\frac{p^{r_{ij}} - p^{r_{i-1, j+1}}}{p-1} & \textrm{ when } i < p-1\\
\frac{p^{r_{p-1, j}} - p^{r_{p-2, j+1}}}{p(p-1)} & \textrm{ when } i = p-1.
\end{cases}
\]
\end{cor}

\subsubsection*{\textbf{\emph{Some refined results when \texorpdfstring{$p$}{} is a regular prime}}}

The results we have proven so far do not require the assumption that $p$ is regular.
In the special case when $p$ is a regular prime, one can prove some stronger results.
The main result we prove next is the following.

\begin{Th}
\label{no unramified extn}
Let $p$ be a regular prime.
With notation introduced before, there are no unramified $A^{p-1, j}$ extensions of $L$.
\end{Th}

To prove this theorem we first prove a technical result which is interesting in its own right.
We emphasize that the following result \emph{does not} require the hypothesis that $p$ is a regular prime.

\begin{Th}
\label{th: one dim coh grp}
The cohomology group $H^1_\Lambda(G_{\QQ}, \Fp(1))$ is 1 -dimensional and generated by $b$.
\end{Th}

\begin{proof}
\cite[Theorem~2.3.3(2)]{SS19} asserts that $H^1(G_{\QQ,S}, \Fp(1))$ is 2-dimensional and is generated by $b$ and the cocycle
\begin{align*}
  c: G_{\QQ} &\longrightarrow \Fp(1)\\
  \sigma &\mapsto \frac{\sigma(p^{1/p})}{p^{1/p}}.
\end{align*}
Note that $b$ satisfies the Selmer condition $\Lambda$.
We now show that the cocycle $c$ does not satisfy the Selmer condition by proving that $c$ violates the Selmer condition at $p$.
Consider the map
\[
\Res: H^1(G_{\QQ}, \Fp(1)) \longrightarrow H^1(G_L, \Fp(1)) = \Hom(G_L, \Fp).
\]
Note that $\Res(c)$ determines the extension $L(p^{1/p})/L$, so it is enough to show that this extension is ramified at $p$.
Working locally at $p$, we have the following field diagram
\begin{center}
\begin{tikzpicture}
  \node (Q1) at (0,0) {$\QQ_p(\zeta_p)$};
  \node (Q2) at (2,2) {$L_p = \QQ_p(\zeta_p, N^{1/p})$};
  \node (Q3) at (0,4) {$L'_p = \QQ_p(\zeta_p, N^{1/p}, p^{1/p})$};
  \node (Q4) at (-2,2) {$\widetilde{L}_p = \QQ_p(\zeta_p, p^{1/p})$};

  \draw (Q1)--(Q2);
  \draw (Q1)--(Q4);
  \draw (Q3)--(Q4);
  \draw (Q2)--(Q3);
  \end{tikzpicture}
\end{center}

\underline{Case 1: when $N\equiv 1\pmod{p^2}$.}\newline

Lemma~\ref{p-ramification} implies that $L_p/\QQ_p(\zeta_p)$ is unramified.
But $\widetilde{L}_p/\QQ_p(\zeta_p)$ is a ramified extension which forces that $L'_p/\QQ_p(\zeta_p)$ is ramified.
In turn, $L'_p/L_p$ is ramified as desired.
\newline

\underline{Case 2: when $N\not\equiv 1\pmod{p^2}$.} \newline

Set $N= 1+kp$ such that $p\nmid k$.
Since ramification indices are multiplicative, it suffices to show that $L'_p / \widetilde{L}_p$ is a ramified extension.
If the extension is \emph{not} ramified, it is proven in \cite[Chapter I, Theorem~6.3(ii)]{Gras_CFT} (i.e., a general version of Theorem~\ref{Gras}) that there exists $y\in U^{(1)}_{\widetilde{L}_p}$ such that
$\fp^{p^2} \mid (N-y^p)$.
Write $y = 1 + \fp \theta$ and observe that
\begin{align*}
  N-y^p & = (1+kp) - (1+\fp\theta)^p\\
  & = kp - p\fp\theta - {p\choose{2}}\fp^2\theta^2 - \ldots.
\end{align*}
Since $(p) = \fp^{p(p-1)}$,
\[
\val_{\fp}(N-y^p) = p(p-1) < p^2.
\]
But this is a contradiction.
Hence, $L'_p/\widetilde{L}_p$ is a ramified extension.

This completes the proof of the theorem.
\end{proof}

In the following discussion, we impose the condition that $p$ is regular.

\begin{proof}[Proof of Theorem~\ref{no unramified extn}]
Suppose there exists such an unramified extension.
By our earlier discussion (see, for example proof of Theorem~\ref{main theorem Rusiru note 3}) we have a representation 
\[
\rho: G_{\QQ, S} \longrightarrow \begin{pmatrix}
  \cG & \underline{A}\\
  0 & 1
\end{pmatrix} \subseteq \GL_{p+1}(\Fp),
\]
defined in exactly the same way. By looking at the lower $2\times 2$ corner of the above matrix, we get a non-zero class $a_{p-1}\in H^1(G_{\QQ,S}, \Fp(j))$. \newline

\underline{Case 1: when $j=0$.} \newline

We arrive at a contradiction by using \cite[proof of Lemma~3.1.10]{SS19}.
\newline

\underline{Case 2: when $j\neq 0,1$.} \newline 

In this situation, $a_{p-1}$ cuts out an $\Fp(j)$-extension $E$ of $K=\QQ(\zeta_p)$.
The following field diagram will simplify the proof
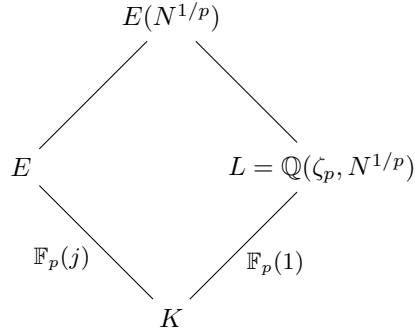
\begin{center}
\begin{tikzpicture}
  \node (Q1) at (0,0) {$K$};
  \node (Q2) at (2,2) {$L = \QQ(\zeta_p, N^{1/p})$};
  \node (Q3) at (0,4) {$E(N^{1/p})$};
  \node (Q4) at (-2,2) {$E$};

  \draw (Q1)--(Q2) node [pos=0.8, below ,inner sep=0.5cm] {\small{$\Fp(1)$}};
  \draw (Q1)--(Q4) node [pos=0.8, below ,inner sep=0.5cm] {\small{$\Fp(j)$}};
  \draw (Q3)--(Q4);
  \draw (Q2)--(Q3);
  \end{tikzpicture}
\end{center}
Note that $E(N^{1/p})$ is inside the initial $A^{p-1,j}$-extension; hence, $E(N^{1/p})/L$ is unramified at $p$.
It follows from \cite[Lemma~3.1.4]{SS19} that $E/K$ is unramified at $p$.
Moreover, $a_{p-1}\in H^1(G_{\QQ,S}, \Fp(j))$ which means that $E/K$ is unramified outside $S$.
Since $p$ is a regular prime (by assumption), we can conclude from class field theory that $E/K$ can not be unramified everywhere. Hence, it is tamely ramified at $N$.
We now get a contradiction as in \cite[proof of Lemma~3.1.10]{SS19}.
\newline

\underline{Case 3: when $j=1$.}\newline

This means $A^{p-1, 1}$ has a quotient $\Fp(1)$, i.e., $L$ has an unramified $\Fp(1)$ extension which is impossible by Theorem~\ref{th: one dim coh grp} and Corollary \ref{define rij}.

This completes the proof of the theorem.
\end{proof}

\subsection{\texorpdfstring{An explicit description of the $p$-rank of $\Cl(L)$}{}}
The main theorem that we prove in this section is the following:

\begin{Th}
\label{main theorem of note 4}
With notation introduced in Corollary~\ref{define rij},
\[
\rk_p(\Cl(L)) = \sum_{j=0}^{p-2}r_{p-1, j}.
\]
Furthermore, when $p$ is a regular prime,
\[
\rk_p(\Cl(L)) = \sum_{j=0}^{p-2}r_{p-2, j}.
\]
\end{Th}

To prove this result we first write
\begin{equation}
\label{introduce mij}
\Cl(L) \otimes \Fp \cong \bigoplus_{i,j} (A^{i,j})^{m_{ij}} \textrm{ as a } \cG\textrm{-module}.    
\end{equation}
Observe that, by \cite[Theorem 3.1.6]{SS19} any $\cG$-module can be written in this way. The key idea of this proof is to compute the number of $A^{i_0, j_0}$ quotients of $\Cl(L) \otimes \Fp$ in terms of $\{m_{ij}\}$.
Then we relate the count to Theorem~\ref{define rij} to get more information about $\{m_{ij}\}$.
The precise relationship between the multiplicities $m_{ij}$ and the ranks $r_{ij}$ is clarified in Corollary~\ref{mij in terms of the rij's}.

For a $\cG$-module $X$, write $X^\vee = \Hom(X, \Fp)$.

\begin{Lemma}
\label{Rusiru note 4 Lemma 4}
For $i\leq p-1$, 
\[
\left(\Sym^{i}(V) \otimes \Fp(j)\right)^{\vee} \cong \Sym^{i}(V) \otimes \Fp(-i-j).
\]
\end{Lemma}

\begin{proof}
This lemma generalizes \cite[Lemma~2.2.1]{SS19}.
Recall that $A^{i,j}$ are indecomposable $\cG$-representations; see \cite[Theorem~3.1.6]{SS19}.
The dual of an indecomposable representation is also indecomposable, so $(A^{i,j})^\vee = A^{s,t}$ for some $s,t$ by \cite[Theorem~3.1.6]{SS19}.
Considering dimensions, note that $i=s$.
Consider the perfect pairing
\[
A^{i,j} \times (A^{i,j})^\vee \longrightarrow \Fp.
\]
Now, $\Fp(i+j) \cong A^{0, i+j}\subseteq A^{i,j}$ and its annihilator in $A^{i, t}$ must be $i$-dimensional.
We proved earlier that the unique $i$-dimensional subrepresentation of $A^{i, t}$ is $A^{i-1, t+1}$.
We have the perfect pairing
\[
\Fp(i+j) \times \frac{A^{i,t}}{A^{i-1, t+1}} \cong A^{0,t} \longrightarrow \Fp.
\]
Recall that $\frac{A^{i,t}}{A^{i-1, t+1}} \cong A^{0,t} \cong \Fp(t)$.
The perfect pairing is possible only when $t=-(i+j)$.
\end{proof}

In view of the above lemma, the number of $A^{i_0, j_0}$ quotients of $\Cl(L)\otimes \Fp$ is the same as the number of $A^{i_0, -i_0-j_0}$-subrepresentations of $(\Cl(L)\otimes \Fp)^{\vee}$.
So, let $X \cong \bigoplus_{i,j} (A^{i,j})^{n_{ij}}$ be a $\cG$-representation and we now count the $A^{i_0, j_0}$-subrepresentations of $X$.

\begin{definition}
A \emph{principal generator} of $A^{i_0, j_0}$ is a vector $v$ such that $g\cdot v = \chi^{j_0}(g)v$ for all $g\in G' = \Gal(L/F)$ where $F =\QQ(N^{1/p})$ and $v$ generates $A^{i_0, j_0}$ as a $\cG$-module. 
\end{definition}

The purpose of the next two results is to count the number of principal generators of $A^{i_0, j_0}$.

\begin{prop}
\label{Rusiru note 4 Prop 5}
When $i_0 < p-1$, the $\cG$-module $A^{i_0, j_0}$ has $p-1$ principal generators.
\end{prop}

\begin{proof}
Pick a basis of $A^{i_0, j_0}$ such that the $\cG$-action is give by the matrix $[\rho_{i_0, j_0}](g)$ where $g\in G_{\QQ}$.
For all $g'\in G'$, the action of $g'$ is given by the diagonal matrix $D_{i_0, j_0} := \textrm{Diag}(\chi^{i_0 + j_0}(g'), \; \ldots, \; \chi^{j_0}(g'))$.
Note that the first condition to be a principal generator is satisfied exactly by (the set of) column vectors $\{(0, \; \ldots, 0, \; x)^{T} \ \mid \ x\in \Fp^\times\}$.
\newline

\emph{Claim:} For all $i,j$, the set of column vectors $\{(*, \; \ldots, *, \; 1)^{T}\}$ generates $A^{i,j}$ as a $\cG$-representation.
\newline

\emph{Justification:} The proof of the claim is by induction.
There is nothing to prove when $i=0$ and assume that the result is true for all $k< i$.
For each element $g\in G=\Gal(L/K)$,
\[
([\rho_{ij}](g)-\Id)\cdot\begin{pmatrix}
  *\\
  \vdots\\
  \alpha\\
  1
\end{pmatrix} = \begin{pmatrix}
  *\\
  \vdots\\
   b(g)\\
  0
\end{pmatrix} 
\]
Let $g$ be such that $ b(g)\neq 0$.
By induction hypothesis, the column vector $(*, \; \ldots, *, \; b(g), \; 0)^{T}$ generates $A^{i-1, j+1} \subseteq A^{i,j}$.
This together with the vector $(*, \; \ldots, *, \; 1)^{T}$ generates $A^{i,j}$.
\newline

We now conclude that $\{(0, \; \ldots, 0, \; x)^{T} \ \mid \ x\in \Fp^\times\}$ is precisely the set of principal generators and this set has $p-1$ many elements.
\end{proof}

\begin{prop}
\label{Rusiru note 4 Prop 6}
The $\cG$-module $A^{p-1, j_0}$ has $p(p-1)$ many principal generators.
\end{prop}

\begin{proof}
The proof is similar to that of Proposition~\ref{Rusiru note 4 Prop 5}.
More precisely, observe that the condition to be a principal generator is satisfied by (the set of) column vectors $\{(y, \; 0 \; \ldots, 0, \; x)^{T} \ \mid \ x,y\in \Fp\}$.
But in view of the claim (in the proof of the previous proposition) we can generate $A^{p-1, j_0}$ when $x\in \Fp^\times$.
Thus there are $p(p-1)$ many principal generators.
\end{proof}

\begin{rem}
\label{some remark}
\leavevmode{}
\begin{enumerate}
\item[(a)] \label{}The elements of the centralizer acts transitively on the set of principal generators.
\item[(b)] The set of column vector $\{(0, \ \ldots, \ x, \ 0, \ \ldots 0 )^{\textrm{T}} \mid x\in \Fp^\times \textrm{ and }x \textrm{ in }k\textrm{-th place} \}\subseteq A^{i,j}$ is the set of principal generators for $A^{k, i+j-k}\subseteq A^{i,j}$.
\end{enumerate}
\end{rem}

\begin{prop}
\label{Rusiru note 4 prop 7}
Fix $0 \leq \theta \leq p-2$.
Let 
\[
v = \sum \gamma_{ij}^{(r)} v_{ij}^{(r)} \in X \cong \bigoplus_{i,j} (A^{i,j})^{n_{ij}}, 
\]
where $\gamma_{ij}^{(r)}$ is in the centralizer (possibly also the zero matrix) described previously and $v_{ij}^{(r)}$ is in the $r$-th copy of $A^{i,j}$.
Moreover, $v_{ij}^{(r)}$ is a column vector with 1 in the $\delta_{ij}^{(r)}$-th position and 0 everywhere else and it is a principal generator of $A^{\delta_{ij}^{(r)}, \theta} \subseteq A^{i,j}$.
Let $\ell(v) = \ell := \max\{ \delta_{ij}^{(r)} \ : \ \gamma_{ij}^{(r)} \neq 0 \}$.
Then 
\[
\cG\cdot v \cong A^{\ell, \theta}.
\]
\end{prop}

\begin{proof}
In view of Remark~\ref{some remark}(a) note that without loss of generality, $\gamma_{ij}^{(r)} \in \{0,I_{i+1}\}$.
The isomorphism in the statement of the proposition is then induced by the map
\[
v \mapsto \begin{pmatrix}
  0\\
  \vdots\\
  0\\
  1
\end{pmatrix}. \qedhere
\]
\end{proof}

Any $v\in X$ satisfying $g'\cdot v = \chi^{\theta}(g)v$ for all $g'\in G'$ are in the above form.
Therefore, it follows that there is a map
\begin{align*}
\{v\in X \ : \ g'\cdot v = \chi^{j_0}(g')v \textrm{ for all } g'\in G' \textrm{ and }j_0=\ell(v)\} &\mapsto \{A^{i_0, j_0} \textrm{ subrepresentations of }X\}\\  
v & \mapsto \cG\cdot v 
\end{align*}
Moreover, the above map is surjective; it is of the form $p-1 : 1$ when $i_0 < p-1$ and of the form $p(p-1) : 1$ when $i_0 = p-1$.
To calculate the number of subrepresentations it suffices to count the number of vectors of the form arising in Proposition~\ref{Rusiru note 4 prop 7}.
We record this in the proposition below.

\begin{prop}
\label{Rusiru note 4 prop 8 and 9}
Define
\begin{align*}
  u_{i_0, j_0} & = \sum_{\alpha = i_0}^{p-1} n_{\alpha, j_0 + i_0 - \alpha} + \sum_{\alpha=0}^{i_0-1} \sum_{\beta = j_0 - \alpha}^{j_0} n_{\alpha\beta} + \sum_{\alpha=i_0}^{p-1} \sum_{\beta = j_0 - \alpha}^{i_0 + j_0 - \alpha - 1} n_{\alpha\beta} \\
  v_{i_0, j_0} & = \sum_{\alpha=0}^{i_0-1} \sum_{\beta = j_0 - \alpha}^{j_0} n_{\alpha\beta} + \sum_{\alpha=i_0}^{p-1} \sum_{\beta = j_0 - \alpha}^{i_0 + j_0 - \alpha - 1} n_{\alpha\beta}.
\end{align*}
Then the number of $A^{i_0, j_0}$ subrepresentations of $X$ is given by
\[
\begin{cases}
 \frac{p^{u_{i_0, j_0}} - p^{v_{i_0, j_0}}}{p-1}  & \textrm{ if } i_0 < p-1\\
 \frac{p^{u_{i_0, j_0}} - p^{v_{i_0, j_0}}}{p(p-1)}  & \textrm{ if } i_0 = p-1.\\
\end{cases}
\]
\end{prop}

\begin{proof}
This is obtained by a careful counting.
\end{proof}

In the following proposition we count the number of $A^{i_0, j_0}$ quotients of $\Cl(L) \otimes \Fp \simeq \bigoplus_{ij} (A^{i,j})^{m_{ij}}$.

\begin{prop}
\label{Rusiru note 4 prop 10}
Define
\[
w_{ij} = 
\begin{cases} \sum_{\alpha=0}^{i} \sum_{\beta = i + j - \alpha}^{i + j} m_{\alpha\beta} + \sum_{\alpha=i + 1}^{p-1} \sum_{\beta = j}^{i + j} m_{\alpha\beta} & \textrm{ when } i \geq 0 \\ 
0 & \textrm{ when } i <0.
\end{cases}
\]
Then the number of $A^{i_0, j_0}$ quotients of $\Cl(L) \otimes \Fp$ is given by
\[
\begin{cases}
 \frac{p^{w_{i_0, j_0}} - p^{w_{i_0 -1, j_0+1}}}{p-1}  & \textrm{ if } i_0 < p-1\\
 \frac{p^{w_{i_0, j_0}} - p^{w_{i_0-1, j_0+1}}}{p(p-1)}  & \textrm{ if } i_0 = p-1.\\
\end{cases}
\]
\end{prop}

\begin{proof}
The proof follows from Proposition~\ref{Rusiru note 4 prop 8 and 9}, Lemma~\ref{Rusiru note 4 Lemma 4}, and the discussion after the lemma.
\end{proof}

We now present the proof of the main theorem of this section.

\begin{proof}[Proof of Theorem~\ref{main theorem of note 4}]
Recall the definition of $r_{ij}$ introduced in Corollary~\ref{define rij}.
For all $i,j$ note that $r_{ij} = w_{ij}$.
Also, for all $i,j$ it follows from our earlier discussions that
\[
w_{ij} - w_{i-1, j+1} = \sum_{\alpha = i}^{p-1} m_{\alpha j}.
\]
Putting all of this together
\begin{align*}
  \rk_p(\Cl(L)) & = \sum_{i,j} (i+1)m_{ij}\\
  & = \sum_{j=0}^{p-2} \sum_{i=0}^{p-1} \sum_{\alpha=i}^{p-1} m_{\alpha j}\\
  & = \sum_{j=0}^{p-2} \sum_{i=0}^{p-1} \left(w_{ij} - w_{i-1, j+1}\right)\\
  & = \sum_{j=0}^{p-2} w_{p-1, j}\\
  & = \sum_{j=0}^{p-2} r_{p-1, j}.
\end{align*}
With this the proof is now complete.

When $p$ is regular, the claim is immediate from Corollary \ref{cor1} and Theorem~\ref{no unramified extn}.
\end{proof}

From the proof of Theorem~\ref{main theorem of note 4} we now derive a precise formula for the $m_{ij}$ introduced in \eqref{introduce mij}.

\begin{cor}
\label{mij in terms of the rij's}
For all $i,j$, the multiplicities $m_{ij} = r_{ij} + r_{i, j+1} - r_{i-1, j+1} - r_{i+1, j}$.    
\end{cor}

\begin{proof}
It follows from the proof of Theorem~\ref{main theorem of note 4} that
\begin{align*}
r_{ij} - r_{i-1, j+1} & = \sum_{\alpha=i}^{p-1} m_{\alpha j}\\
r_{i+1, j} - r_{i, j+1} & = \sum_{\alpha=i+1}^{p-1} m_{\alpha j}
\end{align*}
Subtracting the two equations we get the desired result.
\end{proof}

When $p=3$, we prove a precise result regarding the $\cG$-module structure of $\Cl(L)\otimes \mathbb{F}_3$.

\begin{cor}
When $p=3$,    
\[
\Cl(L) \otimes \Fp \cong \begin{cases}
    \Fp \textrm{ when } \rk_3(L)=1\\
    V \textrm{ when } \rk_3(L)=2,
\end{cases}
\]
where $V\simeq \Fp^2$ with Galois action as defined in \eqref{rep}.
\end{cor}

\begin{proof}
By Theorem~\ref{main theorem of note 4},
\[
\rk_p(\Cl(L)) = r_{1,0} + r_{1,1}.
\]
But, we also have the following exact sequence
\[
0 \longrightarrow H^1_\Lambda(G_\QQ, \Fp) \longrightarrow \frac{H^1_\Lambda(G_{\QQ}, V)}{\langle \underline{\textbf{b}}^{(2)} \rangle} \longrightarrow \frac{H^1_{\Lambda}(G_\QQ, \Fp(1))}{\langle b \rangle}.
\]
The first term is 1-dimensional by \cite[Remark~3.2.1]{SS19} and the last term is trivial by Theorem~\ref{th: one dim coh grp}.
So, $r_{1,1}=r_{0,0}=1$ and $r_{0,1}=0$.
On the other hand, $r_{1,0}$ is either 0 or 1, depending on $\rk_p(\Cl(L))$.
The result follows from Corollary~\ref{mij in terms of the rij's}.
\end{proof}

\subsection{Improved lower bounds}
\label{sec: revised LB}
The purpose of this section is to use tools from Galois cohomology that we have developed in this section to provide a refinement of Theorem~\ref{lower bound theorem} when $p$ is a regular prime.
We provide an elegant relationship between $\rk_p(\Cl(L))$ and $\rk_p(\Cl(F))$ in Corollary~\ref{cor regular prime lower bound neat}.

Before stating the main result, we introduce another Selmer condition from \cite{SS19}.
\begin{definition}
Let $A$ be a $G_{\QQ}$-module.
Define $\Sigma = \{L_v\}$ to be the \textit{Selmer condition} given by 
\begin{itemize}
  \item$L_{\ell} =H^1_{\ur}(G_{\QQ_{\ell}} ,A) = H^1(G_{\QQ_{\ell}}/I_{\ell}, A^{I_\ell})$ for $\ell \not\in S$.
  \item $L_N = \ker\left( \Res : H^1(G_{\QQ_N} ,A) \rightarrow H^1(G_{F_N} ,A)\right)$ where $F_N = \QQ_N(N^{1/p})$.
  \item $L_p = 0$.
\end{itemize}
The \textit{Selmer group} associated to the Selmer condition $\Sigma$ is defined as
\[
H^1_{\Sigma}(G_{\QQ}, A) = \ker\left(H^1(G_{\QQ},A) \longrightarrow \prod_{v} \frac{H^1(G_{\QQ_v}, A)}{L_v}\right).
\]
\end{definition}

More precisely, we prove the following theorem
\begin{Th}
\label{main result - revised lower bounds}
Let $p$ be a regular prime.
Then
\begin{align*}
\rk_p\left(\Cl(L)\right) & \geq \rk_p\left(\Cl(F)\right) + \frac{p-3}{2} + \sum_{\substack{j=2 \\ j \text{ even}}}^{p-3}\dim_{\Fp}\left( H^1_{\Lambda}(G_{\QQ}, \Fp(j))\right)\\
&\geq \frac{p-1}{2} + \dim_{\Fp}\left( H^1_{\Sigma}(G_{\QQ}, \Fp(-1))\right) +  \sum_{\substack{j=2 \\ j \text{ even}}}^{p-3}\dim_{\Fp}\left( H^1_{\Lambda}(G_{\QQ}, \Fp(j))\right).
\end{align*}
\end{Th}

The proof will occupy the remainder of this section.

\begin{prop}
\label{1st prop Rusiru note better bounds}
Consider the map
\[
\iota: H^1_\Lambda\left( G_{\QQ}, A^{i-1, j+1}\right) \longrightarrow H^1_\Lambda\left( G_{\QQ}, A^{i, j}\right)
\]
induced by the inclusion $A^{i-1, j+1} \subseteq A^{i,j}$.
If $p-1 \nmid j$, then $\iota$ is injective; whereas if $p-1 \mid j$ then $\ker(\iota)$ is 1-dimensional.
\end{prop}

\begin{proof}
Consider the short exact sequence
\[
0 \longrightarrow A^{i-1, j+1} \longrightarrow A^{i,j} \longrightarrow \Fp(j) \longrightarrow 0
\]
and take the $G_{\QQ,S}$ cohomology to obtain
\[
\Fp(j)^{G_{\QQ,S}} \xrightarrow{\theta} H^1(G_{\QQ,S}, A^{i-1, j+1}) \xrightarrow{\iota} H^1(G_{\QQ,S}, A^{i, j}).
\]

When $p-1 \nmid j$, observe that $\Fp(j)^{G_{\QQ,S}} =0$ which implies the first claim.

When $p-1\nmid j$, the generator of $\ker(\iota) = \Image(\theta)$.
A careful diagram chase shows that the kernel is generated by the cocycle $[\uline{\boldsymbol{b}}^{(i)}]$.
\end{proof}

\begin{prop}
\label{2nd prop Rusiru note better bounds}
The following map is injective
\[
\widetilde{\iota} : H^1_\Lambda(G_{\QQ}, A^{i-1,2}) \longrightarrow \frac{H^1_\Lambda(G_{\QQ}, A^{i-1,1})}{\langle \underline{\textbf{b}}^{(i+1)} \rangle}  
\]
\end{prop}

\begin{proof}
Let us first consider the following commutative diagram:
\[ 
\begin{tikzcd}
H^1_\Lambda(G_{\QQ}, A^{i-1,2}) \arrow{r}{\widetilde{\iota}} \arrow[swap]{d}{\Res} & \frac{H^1_\Lambda(G_{\QQ}, A^{i-1,1})}{\langle \underline{\textbf{b}}^{(i+1)} \rangle}  \arrow{d}{\Res} \\%
\Hom(G_L, A^{i-1, 2}) \arrow{r}{}& \Hom(G_L, A^{i, 1})
\end{tikzcd}
\]
By Proposition~\ref{res is inj}, the vertical maps are injective.
Note that the bottom horizontal arrow is injective because it is induced by the injection $A^{i-1,2} \subseteq A^{i,1}$.
This completes the proof.
\end{proof}

We have the following filtration
\begin{footnotesize}
\begin{equation}
\label{filt before prop}
H^1_\Lambda(G_{\QQ}, \Fp(i+j)) \subseteq \ldots \subseteq H^1_\Lambda(G_{\QQ}, A^{i-1, j+1}) \subseteq H^1_\Lambda(G_{\QQ}, A^{i, j}) \subseteq \ldots \subseteq \frac{H^1_\Lambda(G_{\QQ}, A^{i+ j -1, 1}) }{\langle \underline{\boldsymbol{b}}^{(i+j)}\rangle} \subseteq \ldots \subseteq H^1_\Lambda(G_\QQ, A^{p-1, i+j}).
\end{equation}
\end{footnotesize}

With notation as introduced before,
\[
r_{p-1, j} \geq \dim_{\Fp} \left( H^1_\Lambda(G_\QQ, \Fp(j))\right).
\]
Therefore,
\[
\rk_p\left(\Cl(L)\right) \geq r_{p-1, 0} + \sum_{j=2}^{p-2} \dim_{\Fp} \left( H^1_\Lambda(G_\QQ, \Fp(j))\right)
\]

\begin{prop}
With notation as before
\[
r_{p-1,0} = r_{p-3,2} = \rk_p\left(\Cl(\QQ(N^{1/p}))\right)
\]
\end{prop}

\begin{proof}
Consider $x\in H^1_\Lambda(G_\QQ, A^{p-1,0}) \setminus H^1_\Lambda(G_\QQ, A^{p-2,1})$.
This gives a ramified $A^{p-1,0}$ extension of $L$ and a representation
\begin{align*}
\rho_x : G_{\QQ, S} &\longrightarrow \begin{pmatrix}
\cG & \underline{A} \\
0 & 1 \\
\end{pmatrix} \subseteq \GL_{p+1}(\mathbb F_p) \\
g & \mapsto \begin{pmatrix}
  * & \underline{\boldsymbol{a}}_g^{(p-1)}\\
  0 & 1 \\
\end{pmatrix}.
\end{align*}
This contradicts \cite[Lemma~3.1.10]{SS19}. Analogously, if $x\in H^1_\Lambda(G_\QQ, A^{p-2,1}) \setminus H^1_\Lambda(G_\QQ, A^{p-3,2})$ we contradict \cite[Lemma~3.1.11]{SS19}.
The second equality is precisely \cite[Theorem~3.2.2]{SS19}.
\end{proof}

We use the same notation as in \cite{SS19} for defining Selmer condition.
Set $S = \{p, N, \infty\}$ and $T\subseteq S$.
As before, all Selmer conditions discussed here have the unramified condition at places outside of $S$.
For the module $A$, write $H^1_T(G_{\QQ}, A)$ to denote the Selmer group with the unramified condition at all places outside of $T$, and any behaviour allowed at the places of $T$.

\begin{prop}
\label{prop dimFp H1Lambda = 1}
Let $p$ be a regular prime and $j \not\equiv 1 \pmod{p-1}$ be odd.
Then, 
\[
\dim_{\Fp}(H^1_\Lambda\left( G_{\QQ}, \Fp(j)\right)) =1.
\]
\end{prop}

\begin{proof}
We first begin with the observation that for the $G_{\QQ}$-module $\Fp(j)$ we have the following inclusions of the Selmer groups associated to corresponding Selmer conditions:
\[
H^1_N(G_{\QQ}, \Fp(j)) \subseteq H^1_{\Lambda}(G_{\QQ}, \Fp(j)) \subseteq H^1(G_{\QQ,S}, \Fp(j)).
\]
It follows from \cite[Theorem~2.3.5]{SS19} that $\dim_{\Fp} \left(H^1_{\Lambda}(G_{\QQ}, \Fp(j))\right)$ is 1 or 2 when $j \not\equiv 1\pmod{p-1}$.
Suppose that the dimension is 2, then
\begin{equation}
  \label{eqn star}
H^1_{p}(G_{\QQ}, \Fp(j)) \subseteq H^1_{S}(G_{\QQ}, \Fp(j)) = H^1_{\Lambda}(G_{\QQ}, \Fp(j)).
\end{equation}
Consider a non-zero element $x\in H^1_{p}(G_{\QQ}, \Fp(j))$.
Since $p$ is a regular prime, $\Res_p(x) \in H^1(G_{\QQ_p}, \Fp(j))$ is non-zero and defines a ramified extension $\cF/\QQ_p(\zeta_p)$ such that $\Gal(\cF/\QQ_p(\zeta_p)) \simeq \Fp(j)$ as a $\Gal(\QQ_p(\zeta_p)/\QQ_p)$-representation.
Consider the following diagram

\begin{center}
\begin{tikzpicture}
  \node (Q1) at (0,0) {$\QQ_p(\zeta_p)$};
  \node (Q2) at (2,2) {$L_p = \QQ_p(\zeta_p, N^{1/p})$};
  \node (Q3) at (0,4) {$\cF' = \cF L_p$};
  \node (Q4) at (-2,2) {$\cF$};

  \draw (Q1)--(Q2) node [pos=0.8, below ,inner sep=0.5cm] {\small{$\Fp(1)$}};
  \draw (Q1)--(Q4) node [pos=0.8, below,inner sep=0.5cm] {\small{$\Fp(j)$}};
  \draw (Q3)--(Q4);
  \draw (Q2)--(Q3);
  \end{tikzpicture}
\end{center}

\emph{Claim:} $\cF'/L_p$ is a ramified extension.
\newline

\emph{Justification:} If the extension $\cF'/L_p$ is unramified, then \cite[Lemma~3.1.4]{SS19} asserts that $\cF/\QQ_p(\zeta_p)$ is also unramified but this is a contradiction.
This proves the claim.
\newline

Since $\cF L_p/L_p$ is ramified it implies that
\[
\Res_p(x) \notin H^1_{\ur}\left( G_{L_p}, \Fp(j)\right).
\]

This contradicts \eqref{eqn star} which means that in our setting $\dim_{\Fp}\left(H^1_\Lambda(G_\QQ, \Fp(j)) \right)=1$.
\end{proof}

We are now in a position to prove the main result.
\begin{proof}[Proof of Theorem~\ref{main result - revised lower bounds}]
The proof of the theorem follows from Proposition~\ref{prop dimFp H1Lambda = 1}.

The second inequality follows from the fact proven in \cite{SS19} that
\[
\rk_p(\Cl(F)) \geq 1 + \dim_{\Fp}\left( H^1_{\Sigma}(G_{\QQ}, \Fp(-1))\right).
\]
\end{proof}

An immediate corollary of the main theorem is the case when $p=5$.

\begin{cor}
\label{p=5 cor for lb}
Let $p=5$.
Then
\[
\rk_5\left(\Cl(L)\right) \geq \begin{cases}
  2 & \textrm{ if } \rk_5\left(\Cl(F)\right)=1\\
  3 & \textrm{ if } \rk_5\left(\Cl(F)\right)=2\\
  6 & \textrm{ if } \rk_5\left(\Cl(F)\right)=3.
\end{cases}
\]
\end{cor}

\begin{proof}
Recall from \cite[Theorem~6.2.1]{SS19} that
\[
\rk_5\left(\Cl(F)\right) =1+\dim_{\mathbb{F}_5}\left(H^1_\Sigma(G_\QQ, A^{1,2})\right)= 1 + \dim_{\mathbb{F}_5}\left(H^1_\Sigma(G_\QQ, \mathbb{F}_5(-1)) \right) + \dim_{\mathbb{F}_5}\left(H^1_\Sigma(G_\QQ, \mathbb{F}_5(-2)) \right).
\]
Observe that
\begin{small}
\begin{align*}
\rk_5\left(\Cl(L)\right) & = \rk_5\left(\Cl(F)\right) + r_{4,1} + r_{4,2} + r_{4,3} \textrm{ by Theorem~\ref{main theorem of note 4}}\\
& \geq \rk_5\left(\Cl(F)\right) + r_{0,2} + r_{1,2} \textrm{ using the filtration on $A^{i,j}$ and ignoring }r_{4,1}\\
&\geq \left(1 + \dim_{\mathbb{F}_5}\left(H^1_\Sigma(G_\QQ, \mathbb{F}_5(-1)) \right) + \dim_{\mathbb{F}_5}\left(H^1_\Sigma(G_\QQ, \mathbb{F}_5(-2)) \right) \right)\\ 
& \qquad + \dim_{\mathbb{F}_5}\left(H^1_\Sigma(G_\QQ, \mathbb{F}_5(-2)) \right) + \dim_{\mathbb{F}_5}\left(H^1_\Sigma(G_\QQ, A^{1,2}) \right)\\
&\geq \left(1 + \dim_{\mathbb{F}_5}\left(H^1_\Sigma(G_\QQ, \mathbb{F}_5(-1)) \right) + \dim_{\mathbb{F}_5}\left(H^1_\Sigma(G_\QQ, \mathbb{F}_5(-2)) \right) \right)\\ 
& \qquad + \dim_{\mathbb{F}_5}\left(H^1_\Sigma(G_\QQ, \mathbb{F}_5(-2)) \right) + \left( \dim_{\mathbb{F}_5}\left(H^1_\Sigma(G_\QQ, \mathbb{F}_5(-1)) \right) + \dim_{\mathbb{F}_5}\left(H^1_\Sigma(G_\QQ, \mathbb{F}_5(-2)) \right) \right) \\
&\geq 1 + 2 \dim_{\mathbb{F}_5}\left(H^1_\Sigma(G_\QQ, \mathbb{F}_5(-1)) \right) + 3 \dim_{\mathbb{F}_5}\left(H^1_\Sigma(G_\QQ, \mathbb{F}_5(-2)) \right).
\end{align*}
\end{small}
Finally, recall that $\dim_{\mathbb{F}_5}\left(H^1_\Sigma(G_\QQ, \Fp(-1)) \right)$ and $\dim_{\mathbb{F}_5}\left(H^1_\Sigma(G_\QQ, \Fp(-2)) \right)$ are either 0 or 1 and this determines $\rk_5\left( \Cl(F)\right)$.

When $\rk_5(\Cl(F))=1$ the claimed inequality follows from the trivial bounds obtained previously.
When $\rk_5(\Cl(F))=2$, the work of Schaefer-Stubley guarantees that 
\[
\dim_{\mathbb{F}_5}\left(H^1_\Sigma(G_\QQ, \mathbb{F}_5(-1)) \right) = 1 \textrm{ and } \dim_{\mathbb{F}_5}\left(H^1_\Sigma(G_\QQ, \mathbb{F}_5(-2)) \right) = 0.
\]
Finally, when $\rk_5(\Cl(F))=3$, it is clear that
\[
\dim_{\mathbb{F}_5}\left(H^1_\Sigma(G_\QQ, \mathbb{F}_5(-1)) \right) = \dim_{\mathbb{F}_5}\left(H^1_\Sigma(G_\QQ, \mathbb{F}_5(-2)) \right) =1.
\]
The inequalities now follow immediately.
\end{proof}

Next we prove a corollary which gives a relationship between $\rk_p(\Cl(F))$ and $\rk_p(\Cl(L))$.
The following statement is written for $p\geq 7$ as that is the more interesting case, but note that it is also true for $p=3,5$. 

\begin{cor}
\label{cor regular prime lower bound neat}
Let $p$ be a regular prime.
Then
\[
\rk_p(\Cl(L)) \geq 2\rk_p(\Cl(F)) + \frac{p-7}{2}.
\]
\end{cor}

\begin{proof}
By \cite[Proposition~3.3.1]{SS19},
\begin{align*}
\rk_p(\Cl(F)) &= 1 + \dim_{\Fp}\left(H^1_{\Sigma}\left(G_\QQ, A^{p-4,2}\right)\right)\\
&\leq 1 + \dim_{\Fp}\left(H^1_{\Lambda}\left(G_\QQ, A^{p-4,2}\right)\right)\\
&\leq 1 + \dim_{\Fp}\left(H^1_{\Lambda}\left(G_\QQ, A^{p-2,0}\right)\right) \quad \textrm{by }\eqref{filt before prop}\\
&= 1 + r_{p-2,0}.
\end{align*}
Moreover the filtration also shows that $r_{p-2,j}\geq r_{0,j-1}$.
By Theorem~\ref{main theorem of note 4}
\begin{align*}
\rk_p(\Cl(L)) & = \sum_{j=0}^{p-2} r_{p-2,j}  = r_{p-2,0} + \rk_p(\Cl(F)) + \sum_{j=2}^{p-2} r_{p-2,j}    \\
&\geq \left(\rk_p(\Cl(F))-1\right) + \rk_p(\Cl(F)) + \sum_{j=1}^{p-3} r_{0,j} \textrm{ from above discussion}\\
&\geq 2\rk_p(\Cl(F)) -1 + \left( \frac{p-3}{2} -1\right) \textrm{ by Proposition~}\ref{prop dimFp H1Lambda = 1}. \qedhere 
\end{align*}
\end{proof}

We now record a corollary regarding the structure of $\Cl(L)\otimes \Fp$ in a special case.

\begin{cor}
When $p$ is regular,
\[
\rk_p(\Cl(L)) \geq \frac{p-1}{2}.
\]
Moreover, when equality occurs the following isomorphism is true (as a $\cG$-module)
\[
\Cl(L) \otimes \Fp \simeq \bigoplus_{\substack{j=3\\odd}}^{p-2}\Fp(j) \oplus \Fp.
\]
\end{cor}

\begin{proof}
The first statement is immediate from Theorem~\ref{main result - revised lower bounds} and the fact that
\[
\rk_p(\Cl(F)) = r_{p-2,1} \geq r_{0,0} = 1.
\]

The filtration in \eqref{filt before prop} and Theorem~\ref{main theorem of note 4} imply that equality is possible when $r_{0,0} = r_{0,j}=1$ for odd $j\not\equiv 1 \pmod{p-1}$ and all the other $r_{ij}=0$.
The second assertion follows from Corollary~\ref{mij in terms of the rij's}.
\end{proof}

\subsection{Improved upper bounds}
\label{sec: revised UB}
We prove an analogue of \cite[Theorem~3.0.1]{SS19} which provides refined estimates of the upper bound for the $p$-rank of $\Cl(L)$ in the case that $p$ is a regular prime.
The statement of the main theorem is the following.

\begin{Th}
\label{th: better upper bounds}
Let $p$ be a regular prime.
Then
\[
\rk_p(\Cl(L)) \leq \frac{3p-5}{2} + (p-2)\sum_{i=2}^{p-2}\dim_{\Fp}\left( H^1_\Sigma(G_\QQ, \Fp(i))\right) + \sum_{\substack{i=2\\ \textrm{even}}}^{p-3}\dim_{\Fp}\left( H^1_\Lambda(G_\QQ, \Fp(i))\right).
\]
\end{Th}

\begin{proof}
By \cite[Lemma~3.3.2]{SS19} the following exact sequence{\footnote{In the reference the result is claimed for $i\leq p-3$, but the proof works for $i=p-2$, as well.}} exists
\[
0 \longrightarrow H^1_\Lambda(G_\QQ, A^{i-1, j+1}) \longrightarrow H^1_\Lambda(G_\QQ, A^{i, j}) \longrightarrow H^1_{\Lambda\cap \Sigma^*}(G_\QQ, \Fp(j))
\]
when $j\neq 0,1$.
[Here we use the notation $\Sigma^*$ to denote the dual Selmer condition, i.e., $\Sigma^* = \{L_v^{\perp}\}$ where $L_v^\perp$ is the annihilator of $L_v$ under the local cup product pairing.]
On the other hand, when $j=0,1$ Propositions~\ref{1st prop Rusiru note better bounds} and \ref{2nd prop Rusiru note better bounds} imply that 
\begin{equation}
\label{eqn: two ses for j = 0,1}
\begin{split}
0 \longrightarrow H^1_\Lambda(G_\QQ, A^{i-1, 2}) & \longrightarrow \frac{H^1_\Lambda(G_{\QQ}, A^{i,1})}{\langle \underline{\textbf{b}}^{(i+1)} \rangle} \longrightarrow \frac{H^1_{\Lambda}(G_\QQ, \Fp(1))}{\langle b \rangle}\\
0 \longrightarrow \frac{H^1_\Lambda(G_\QQ, A^{i-1, 1})}{\langle \underline{\textbf{b}}^{(i)} \rangle} & \longrightarrow H^1_\Lambda(G_{\QQ}, A^{i,0}) \longrightarrow H^1_{\Lambda}(G_\QQ, \Fp).
\end{split}
\end{equation}

Observe that $\Sigma \subseteq \Lambda \cap \Sigma^*$ and the only difference arises at the place $p$.
\newline

\emph{Claim:} When $j\neq 0,1$, the equality $H^1_\Sigma(G_\QQ, \Fp(j)) = H^1_{\Lambda \cap \Sigma^*}(G_\QQ, \Fp(j))$ holds.
\newline

\emph{Justification:}
We only need to check the condition at $p$.
One inclusion is automatic and we only need to check the other one.
Suppose that $x\in H^1_{\Lambda \cap \Sigma^*}(G_\QQ, \Fp(j))$ is a non-zero element.
Then $x$ determines a Galois extension $E/K$ with Galois group isomorphic to $\Fp(j)$ as a $\Gal(K/\QQ)$-module.
Since $x\in H^1_{\Lambda}(G_\QQ, \Fp(j))$, we know that $E(N^{1/p})/L$ is unramified at $p$.
It follows from \cite[Lemma~3.1.4]{SS19} that $E/K$ is unramified at $p$.
Finally, we can deduce from \cite[Lemma~2.2.5]{Sch_thesis} that $\Res_p(x)\in H^1_{\ur}(G_{\Qp}, \Fp(j))=0$, which in turn implies that $x\in H^1_\Sigma(G_\QQ, \Fp(j))$.
\newline

We know that $\dim_{\Fp}(H^1_{\Lambda}(G_{\QQ}, \Fp))=1$, see \cite[Remark~3.2.1]{SS19}.
This fact combined with Theorem~\ref{th: one dim coh grp} and \eqref{eqn: two ses for j = 0,1} allows us to conclude that 
\begin{align*}
r_{p-2,j+1} & \leq \dim_{\Fp}(H^1_{\Lambda}(G_{\QQ}, \Fp(j))) + \sum_{\substack{i=2 \\ i\neq j}}^{p-2} \dim_{\Fp}(H^1_{\Sigma}(G_{\QQ}, \Fp(i))) + 1 & \textrm{ when } j\neq 0,1,\\
r_{p-2,j} & \leq 1 + \sum_{i=2}^{p-2} \dim_{\Fp}(H^1_{\Sigma}(G_{\QQ}, \Fp(i))) & \textrm{ when } j= 0,1.
\end{align*}
To get the final expression in the theorem, recall that $\dim_{\Fp} H^1_\Lambda(G_{\QQ}, \Fp(j))=1$ for odd $j\neq 1$ which was shown in Proposition~\ref{prop dimFp H1Lambda = 1} combined with  Theorem~\ref{main theorem of note 4}.
\end{proof}

\begin{rem} \label{Rem} \leavevmode{}
\label{explicit count of dimensions}
\begin{enumerate}[(a)]
\item Note that $\dim_{\Fp} H^1_\Sigma(G_{\QQ}, \Fp(i))=0$ or $1$, and the explicit conditions for each case is calculated in \cite[Section~5]{SS19}.
\item For even $i>0$, we have that $\dim_{\Fp} H^1(G_{\QQ,S}, \Fp(i))=1$ by \cite[Theorem~2.3.5(3)]{SS19}.
Therefore, $\dim_{\Fp} H^1_{\Lambda}(G_{\QQ}, \Fp(i))=0$ or $1$.
Theorem~\ref{th: better upper bounds} then guarantees
\[
\rk_p(\Cl(L)) \leq \frac{3p-5}{2} + (p-2)(p-3) + \frac{p-3}{2} = (p-1)(p-2)
\]
which matches with the bound obtained using class field theory.
\item When $\rk_p(\Cl(L)) = (p-1)(p-2)$, all sequences in the proof of Theorem~\ref{th: better upper bounds} are short exact.
For all $j\not \equiv 1\pmod{p-1}$, 
\[
m_{p-2,j} =(r_{p-2,j} - r_{p-3, j+1}) - (r_{p-1,j} - r_{p-2, j+1}) = 1-0 =1.
\]
Here for the penultimate equality we are using the first sequence in \eqref{eqn: two ses for j = 0,1} and the fact that $r_{p-1,j}=0$ when $p$ is regular. 
Therefore, as  $\cG$-module
\[
\Cl(L) \otimes \Fp \simeq \bigoplus_{j=0}^{p-3} \Sym^{p-2}(V) \otimes \Fp(-j).  
\]
\end{enumerate}  
\end{rem}

In the remainder of the section, we calculate $\dim_{\Fp} H^1_{\Lambda}(G_{\QQ}, \Fp(i))$ explicitly.

\begin{Lemma}
\label{lemma for last theorem}
Let $p$ be a regular prime and $i\not\equiv 0\pmod{p-1}$ be even.
Then
\[
\dim_{\Fp}\left( H^1_{\Lambda}(G_\QQ, \Fp(i))\right) = \dim_{\Fp}\left( H^1_{N^*}(G_\QQ, \Fp(1-i))\right),
\]
where the Selmer condition $N^*$ means classes which are split at $N$, have any behaviour at $p$, and are unramified elsewhere.
\end{Lemma}

\begin{proof}
As observed in the proof of Theorem~\ref{th: better upper bounds}, the classes $H^1_{\Lambda}(G_\QQ, \Fp(i))$ are unramified at $p$.
Therefore,
\[
H^1_{\Lambda}(G_\QQ, \Fp(i)) = H^1_{N}(G_\QQ, \Fp(i)).
\]
Observe that $N^*$ is the dual Selmer condition of the Selmer condition $N$.
Writing $\Lambda=\{ L_v\}$ as in Definition~\ref{Lambda Selmer condn} and using \cite[Theorem~2.1.2]{SS19}
\begin{align*}
  \frac{\abs{H^1_{\Lambda}(G_\QQ, \Fp(i))}}{\abs{H^1_{N^*}(G_\QQ, \Fp(1-i))}} & = \frac{\abs{H^0(G_\QQ, \Fp(i))}}{\abs{H^0(G_\QQ, \Fp(1-i))}} \prod_v \frac{\abs{L_v}}{\abs{H^0(G_{\QQ_v}, \Fp(i))}}\\
  & = \frac{1}{1} \times \frac{\abs{H^1(G_{\QQ_N}, \Fp)}}{\abs{\Fp(i)^{G_{\QQ_N}}}} \times \frac{\abs{H^1_{\ur}(G_{\QQ_p}, \Fp(i))}}{\abs{\Fp(i)^{G_{\QQ_p}}}} \times \frac{\abs{H^1_{\ur}(G_{\QQ_{\mathbb{R}}}, \Fp(i))}}{\abs{\Fp(i)^{G_{\mathbb{R}}}}}\\
  & = \frac{1}{1} \times \frac{p^2}{p} \times \frac{1}{1} \times \frac{1}{p} = 1.
  \qedhere
\end{align*}
\end{proof}

\begin{Th}
\label{last result}
Let $p$ be a regular prime and $i\not\equiv 0\pmod{p-1}$ be even.
Then 
\[
\dim_{\Fp}\left(H^1_\Lambda(G_\QQ, \Fp(-i)) \right) =1 \Longleftrightarrow (1-f)(1-f^2)^{2^i} \ldots (1-f^{p-1})^{(p-1)^i} \in (\mathbb{F}_N^\times)^p,
\]
where $f$ is an element of order $p$ in $\mathbb{F}_N^\times$.
\end{Th}

\begin{proof}
In view of the Lemma~\ref{lemma for last theorem}, it suffices to work with $H^1_{N^*}(G_\QQ, \Fp(i+1))$ which may be viewed as a subset of $H^1_{p}(G_\QQ, \Fp(i+1))$.
By \cite[Theorem~2.3.5(2)]{SS19}, this latter cohomology group is 1-dimensional.
This means $\dim_{\Fp}\left(H^1_{N^*}(G_\QQ, \Fp(i+1))\right)=1$ precisely when a generator $x$ of $H^1_{p}(G_\QQ, \Fp(i+1))$ in fact lies in $H^1_{N^*}(G_\QQ, \Fp(i+1))$.
Recall that $x$ determines an extension $E_x/K$.
The above criterion can be equivalently rephrased as, 
\begin{equation}
\label{criterion for dim 1}
\dim_{\Fp}\left(H^1_{N^*}(G_\QQ, \Fp(i+1))\right)=1 \Longleftrightarrow E_x/K \textrm{ is split at }N.
\end{equation}

By Kummer theory, $E_x = \QQ(\zeta_p, \theta^{1/p})$ for some $\theta\in K^\times$ which is not a $p$-th power and such that $\theta$ is in the $\chi^{-i}$-eigenspace of $\frac{K^\times}{(K^\times)^p}$.
\newline

\emph{Claim:} $\theta = (1-\zeta_p)(1-\zeta_p^2)^{2^i} \ldots (1-\zeta_p^{p-1})^{(p-1)^i}$ is a viable candidate.
\newline

\emph{Justification:} Note that $\theta\in K^\times$. Let us show that $\theta\not\in (K^\times)^p$.
First, observe
\[
p \mid (1^i + 2^i + \ldots (p-1)^i).
\]
As an element of $\frac{K^\times}{(K^\times)^p}$,
\begin{align*}
  \theta & = \left(\frac{1-\zeta_p^2}{1-\zeta_p} \right)^{2^i} \left(\frac{1-\zeta_p^3}{1-\zeta_p} \right)^{3^i} \ldots \left(\frac{1-\zeta_p^{p-1}}{1-\zeta_p} \right)^{(p-1)^i}\\
  & = \zeta_p^\alpha \left( \left(\frac{1-\zeta_p^2}{1-\zeta_p} \right)^{2^i + (p-2)^i} \left(\frac{1-\zeta_p^3}{1-\zeta_p} \right)^{3^i + (p-3)^i} \ldots \left(\frac{1-\zeta_p^{\frac{p-1}{2}}}{1-\zeta_p} \right)^{(\frac{p-1}{2})^i + (\frac{p+1}{2})^i} \right)\\
  & =: \zeta_p^\alpha \theta'.
\end{align*}
Since $\theta'$ is a unit in $\ZZ[\zeta_p]$, it suffices to show that $\theta'$ is not a $p$-th power in $\ZZ[\zeta_p]^\times$.

Set $K^+$ to denote the totally real subfield of $K$, $\ZZ[\zeta_p]^+$ to denote its ring of integer, $C^+$ to denote cyclotomic units, and $h_p^+$ to denote the class number of $K^+$.
Recall that $\ZZ[\zeta_p]^\times = \langle \zeta_p \rangle \cO^\times_{\ZZ[\zeta_p]^+}$; see \cite[Theorem~4.12 and Corollary~4.13]{Was97}.
By \cite[Theorem~8.2]{Was97}
\[
h_p^+ = [\cO^\times_{\ZZ[\zeta_p]^+} : C^+].
\]
In view of the assumption that $p$ is regular, 
\[
p\nmid [\ZZ[\zeta_p]^\times : \langle \zeta_p\rangle C^+].
\]
Next observe that $\langle \zeta_p\rangle C^+$ is generated by the set $B = \{\zeta_p, \gamma_2, \ldots, \gamma_{\frac{p-1}{2}}\}$ where $\gamma_k = \frac{1-\zeta_p^k}{1-\zeta_p}$.
Thus,
\[
\frac{\ZZ[\zeta_p]^\times}{(\ZZ[\zeta_p]^\times)^p} \simeq \frac{\langle \zeta_p\rangle C^+}{(\langle \zeta_p\rangle C^+)^p}.
\]
Both have $\Fp$-dimension equal to $\frac{p-1}{2}$ by the Dirichlet Unit Theorem, which means that the image of the elements of $B$ also form a basis for $\frac{\ZZ[\zeta_p]^\times}{(\ZZ[\zeta_p]^\times)^p}$.
Going back to the description of $\theta'$, note that $p\nmid (2^i + (p-2)^i)$ since $i$ must be even which means that $\theta'$ is not a $p$-th power in $\ZZ[\zeta_p]^\times$, as desired.
Hence, $\theta\not\in (K^\times)^p$ and $\QQ(\zeta_p, \theta^{1/p})/ K$ is a non-trivial extension.

Let $\sigma\in \Gal(K/\QQ)$ such that $\sigma(\zeta_p) = \zeta_p^\kappa$.
Then $\chi(\sigma) = \kappa\in (\ZZ/p\ZZ)^\times$.
Working inside $\frac{K^\times}{(K^\times)^p}$,
\begin{align*}
  \sigma(\theta) = (1-\zeta_p^\kappa)(1-\zeta_p^{2\kappa})^{2i} \ldots (1-\zeta_p^{(p-1)\kappa})^{(p-1)i} = \theta^{\kappa^{-i}} = \theta^{\chi^{-i}(\sigma)}.
\end{align*}
Therefore, $\theta$ lies in the appropriate eigenspace and this completes the proof of the claim.

Theorem~\ref{Gras} implies that $\QQ(\zeta_p, \theta^{1/p})$ is unramified outside $p$ and hence determines a non-zero class of $H^1_p(G_\QQ, \Fp(1+i))$.
In view of the criterion in \eqref{criterion for dim 1}, it suffices to verify if $\theta$ is a $p$-th power in $\QQ_N(\zeta_p) = \QQ_N$.
Moreover, since $\theta\in \ZZ_N^\times$, it is equivalent to show that $\theta$ $p$-th power in $\ZZ_N^\times$.
But,
\[
\frac{\ZZ_N^\times}{(\ZZ_N^\times)^p} \simeq \frac{\mathbb{F}_N^\times}{(\mathbb{F}_N^\times)^p}.
\]
Set $\fn\mid N$.
Then $\theta$ is a $p$-th power $\ZZ_N^\times$ precisely when $\theta$ is a $p$-th power in $\left( \ZZ[\zeta_p]/\fn \right)^\times \simeq \mathbb{F}_N^\times$.
Set $\zeta_p = f$ in $\mathbb{F}_N^\times$, then $f$ has order $p$.
In $\mathbb{F}_N^\times$, it is possible to write
\[
\theta = (1-f)(1-f^2)^{2^i} \ldots (1-f^{p-1})^{(p-1)^i}.
\]
[The argument is independent of the choice of the root of unity $\zeta_p$.]
This completes the proof.
\end{proof}

\begin{cor} \label{lastcor}
Let $p$ be a regular prime and  $i$ vary over even integers in the range $\{1, \ldots, p-2\}$.
Let $f$ be any element of order $p$ in $\mathbb{F}_N^\times$.
For an integer $0 < k < p-1$, define
\[
\mathcal{M}_k = (1-f)(1-f^2)^{2^{k}} \ldots (1-f^{p-1})^{(p-1)^{k}}.
\]
Then
\[
\frac{p-1}{2} + \alpha \leq \rk_p\left(\Cl(L)\right) \leq (p-1)(p-2) - (p-1)\left( \frac{p-1}{2}-1-\alpha\right)
,\]
where $\alpha$ is the number of $i \pmod{p-1}$ which are positive, even, and such that $\mathcal{M}_{p-1-i}$ is a $p$-th power in $\mathbb{F}_N^\times$. 
\end{cor}

\begin{proof}
Note that Theorem \ref{last result} implies that $\alpha$ is independent of the choice of $f$. The lower bound follows by combining Theorems~\ref{main result - revised lower bounds} and \ref{last result}.

To obtain the upper bound, first note that the number of $i\pmod{p-1}$ which are positive, even, and such that $\mathcal{M}_{p-1-i}$ is \emph{not} a $p$-th power in $\mathbb{F}_N^\times$ is given by
\[
\frac{p-1}{2}-1-\alpha.
\]
Since $H^1_{\Lambda}(G_{\QQ}, \Fp(j))=0$ forces that $H^1_{\Sigma}(G_{\QQ}, \Fp(j))=0$, an application of Theorem~\ref{th: better upper bounds} implies 
\begin{align*}
\rk_p(\Cl(L)) & \leq  \frac{3p-5}{2}+ (p-2)\left((p-3) - \left( \frac{p-1}{2} - \alpha - 1\right)\right) + \alpha\\
& = \frac{3p-5}{2} + (p-2)(p-3) -(p-2)\left(\frac{p-1}{2}-1-\alpha \right) + \alpha\\
& = p^2 - \frac{7p}{2} + \frac{7}{2}- (p-1)\left( \frac{p-1}{2}-1-\alpha\right) + \left( \frac{p-1}{2}-1-\alpha\right) + \alpha\\
& = p^2 - 3p + 2 - (p-1)\left( \frac{p-1}{2}-1-\alpha\right) \\
& = (p-1)(p-2) - (p-1)\left( \frac{p-1}{2}-1-\alpha\right). \qedhere
\end{align*}

\end{proof}

\bibliographystyle{amsalpha}
\bibliography{references}

\end{document}